\documentclass[reqno,11pt]{article}

\usepackage{latexsym,color,amsmath,amsthm,amssymb,amscd,amsfonts}

\usepackage{epsfig}
\usepackage{bm,bbm}
\usepackage{amsmath}
\usepackage{graphicx}
\usepackage{latexsym,color,amsmath,amsthm,amssymb,amscd,amsfonts}
\usepackage{xcolor}
\usepackage{latexsym}
\usepackage{url}

\usepackage{dsfont}

\date{}

\newtheorem{theorem}{Theorem}[section]

\newtheorem{corollary}[theorem]{Corollary}
\newtheorem{proposition}[theorem]{Proposition}

\newtheorem{remark}[theorem]{Remark}

\def\1{\mathds{1}}

\numberwithin{equation}{section}

\begin{document}

\title{Blaschke--Santal\'o inequality for many functions and geodesic barycenters of measures
}

\author{Alexander V. Kolesnikov\thanks{Supported by RFBR project
20-01-00432;  Section 7 results  have been obtained with support of  RSF grant No  19-71-30020.
The article was prepared within the framework of the HSE University Basic Research Program} \ and Elisabeth M.  Werner\thanks{Partially supported by  NSF grant DMS-1811146 and by a Simons Fellowship}}

\maketitle

\begin{abstract}
Motivated by the geodesic barycenter problem  from  optimal transportation theory,  
we prove a natural generalization of the Blaschke--Santal\'o inequality and the affine isoperimetric inequalities for many sets and many functions.
We derive from it an entropy bound for the total Kantorovich cost appearing in the barycenter problem.
We also establish a ``pointwise Pr\'ekopa--Leindler inequality" and show a monotonicity property of the multimarginal
Blaschke--Santa\'o functional.

\end{abstract}

 \section{Introduction}

The Blaschke--Santal\'o inequality,  see \cite{AGM, SchneiderBook},  states that every  $0$-symmetric convex body $K$ in $\mathbb{R}^n$
satisfies
$$
{\rm vol_n}(K) 
{\rm vol}_n(K^{\circ}) \le ({\rm vol}_n(B^n_2))^2, 
$$  
where $K^{\circ} = \{y \in \mathbb{R}^n:  \langle x, y \rangle \le 1 \ \forall x \in K\}$ is the polar body of $K$, $B^n_2 = \{x \in \mathbb{R}^n: |x| \le 1\}$ is the 
Euclidean unit ball and  $| \cdot |$  denotes the Euclidean norm on $\mathbb{R}^n$. The left-hand side of this inequality is called the Mahler volume.
The sharp lower bound  for the  Mahler volume is still open in dimensions 4 and higher.   The famous Mahler conjecture suggests that this functional is minimized
by the couple $(B^n_1, B^n_{\infty})$.    
Partial results can be found in,  e.g.,  \cite{FHMRZ, IriyehShibata,  NazarovPetrovRyaboginZvavitch, ReisnerSchuttWerner}.
\par
Here we ask: 
What is a natural generalization of the bounds for the Mahler volume for multiple sets?
While this  is not obvious from the geometric viewpoint,  we suggest  in this paper a reasonable extension,
which is naturally related to a functional counterpart of the  Blaschke--Santal\'o inequality.
\par
The functional Blaschke--Santal\'o inequality was discovered by K.~Ball  \cite{KBallthesis}
and later extended and generalized in \cite{ArtKlarMil}, \cite{FradeliziMeyer2007}, \cite{Lehec}.
In its simplest form it states that for every two measurable even functions  $V, W$ on $\mathbb{R}^n$ we have that 
$$
\int e^{-V(x)} dx \int e^{-W(y)} dy \le (2 \pi)^n,
$$
provided that $V(x) + W(y) \ge \langle x, y \rangle$ and either $0 < \int e^{-V(x)}  < \infty $ or  $0 < \int e^{-W(x)}  < \infty$. 
Equality is attained if and only if $V(x) = |Tx|^2 + c, W(y) = |T^{-1}y|^2 - c$, where $T$ is a positive definite matrix and $c>0$ is a constant.
Interesting links to  optimal transportation theory were noted in \cite{ArtKlarMil} and more recently in \cite{Fathi}. There, it is shown  that 
for probability measures $\mu=f \cdot \gamma, \  \nu= g \cdot \gamma$, where $\gamma$ is the standard Gaussian measure, such that
$\int x f d \gamma  =0$,  the following inequality holds, 
\begin{equation}\label{Fathi-intro}
\frac{1}{2} \, W_2^2(\mu, \nu) \le \text{Ent}_\gamma(\mu) + \text{Ent}_\gamma(\nu)
\end{equation}
and that this inequality is equivalent to the functional  Blaschke--Santal\'o inequality.
Here,  $W_2^2(\mu, \nu)$ is the  $L^2$ Kantorovich distance (see Section \ref{IntegralBounds}  for the definition)
 and 
 $$
\text{Ent}_\gamma(\mu) =  \int f \log f d \gamma
 $$
 is the relative entropy with respect to Gaussian measure.
 Inequality (\ref{Fathi-intro}) is a remarkable strengthening of the Talagrand transportation inequality and the starting point of our paper.
 We refer to, e.g., \cite{BGL} for  Talagrand's inequality and it's fundamental importance in  probability theory. 
 In this  context, please also note a very recent result of N.~Gozlan about a transportational approach to the lower bound for the functional Blaschke--Santal\'o 
inequality \cite{Gozlan2020}.
\par
We would like to point out an important connection of the  Blaschke--Santal\'o inequality to the K\"ahler--Einstein equation. 
Inequality (\ref{Fathi-intro}) implies, in particular, that the functional $\mu \to \frac{1}{2} W_2^2(\mu, \nu) - \text{Ent}_\gamma(\mu)$
 is bounded from above. The minimum of this functional solves the so-called  K\"ahler--Einstein  equation. This was established by F.~Santambrogio
 \cite{Santambrogio}. The form of the functional  presented here was considered in \cite{KolesnikovKosov}. The  well-posedness of the K\"ahler--Einstein equation was proved 
by  D.~Cordero-Erausquin and B.~Klartag \cite{CorderoKlartag}. Generalization to the sphere  and relations to the logarithmic Minkowski problem
were established in  \cite{Kolesnikov}. Other related transportation inequalities can be found in \cite{FGJ}.
 \vskip 2mm
 To analyze the case of $k$ functions with $k>2$ we consider the cost functional 
 \begin{equation}
 \label{quadratic}
c(x_1, \cdots, x_n) = \sum_{i,j=1, i<j}^n |x_i - x_j|^2
 \end{equation}
 and the corresponding multimarginal Monge--Kantorovich problem, i.e.,  the minimization problem 
 $$
 P \to \int c\,  d P, \ \ \ P \in \mathcal{P}(\mu_1, \cdots, \mu_k)
 $$ among the measures $\mathcal{P}(\mu_1, \cdots, \mu_k)$ with fixed projections $\mu_1, \cdots, \mu_k$. This problem has been studied by Gangbo and 
 \c{S}wiech
 \cite{GangboSw}. Agueh and Carlier realized in \cite{AguehC} that this problem is naturally related to the barycenter problem for $\mu_1, \cdots, \mu_k$.
A measure ${\mu}$ is called geodesic (or Wasserstein) barycenter of $\mu_1, \cdots, \mu_k$ with coefficients $\frac{1}{k}$, 
if it gives the minimum to the functional $\nu \to \sum_{i=1}^{k} \frac{1}{2k} W_2^2 (\mu_i,\nu) $.
Barycenters of measures have attracted much attention, also among  applied scientists.  We refer to the recent book of Peyr{\`e} and Cuturi \cite{Cuturi} and the references therein
for more information.
\vskip 2mm
\noindent
Motivated by these results  we conjecture that
\begin{equation}
\label{fgh}
 \prod_{i=1}^k \int_{\mathbb{R}^n} f_i(x_i) dx_i
\le \Bigl( \int_{\mathbb{R}^{n}}  \rho^{\frac{1}{k}} \Bigl( \frac{k(k-1)}{2}|u|^2 \Bigr) du\Bigr)^k, 
\end{equation}
where $f_i \colon \mathbb{R}^n \to \mathbb{R}_{+}$, $1 \le i \le k$, are even, measurable, integrable functions satisfying 
$$
\prod_{i=1}^k f_i(x_i)\le  \rho \left( \sum_{i, j=1, i < j}^{k} \langle x_i, x_j \rangle \right)
$$
and  $\rho$ is a  positive non-increasing function.
We verify this conjecture in several cases.  Some of our main results are stated next.

\subsection{The main results}

In Section \ref{IntegralBounds} we discuss some preliminary facts about Kantorovich duality  theory for many functions  and prove that our integral functional is bounded for the case of quadratic cost (\ref{quadratic}).
We also show that for $k>2$ our functional has a trivial (zero) lower bound, unlike the case of two functions. 
\vskip 3mm
\noindent
In Section \ref{UCC} we verify the above conjecture  in the unconditional case (see Section \ref{UCC} for the definition)  and prove the following theorem.
\vskip 2mm 
\noindent
{\bf Theorem  \ref{BSunc}}
{\em Let $f_i \colon \mathbb{R}^n \to \mathbb{R}_{+}$, $1 \le i \le k$, be unconditional integrable functions satisfying 
$$
\prod_{i=1}^k f_i(x_i)\le  \rho \left( \sum_{\substack{i,j=1\\ i<j}}^{k} \langle x_i, x_j \rangle \right) \hskip 2mm \text{for every} \hskip 2mm x_i, x_j \in \mathbb{R}_+^n,  
$$
where $\rho$ is a {positive}  non-increasing function on $[0, \infty)$ such that $\int_{\mathbb{R}}\rho^{\frac{1}{k}}(t^2) dt <\infty$. Then
$$
 \prod_{i=1}^k \int_{\mathbb{R}^n} f_i(x_i) dx_i
\le \left( \int_{\mathbb{R}^{n}}  \rho^{\frac{1}{k}} \left( \frac{k(k-1)}{2}|u|^2 \right) du\right)^k.
$$
For $k>2$, 
equality holds in this  inequality if and only if  there exist positive constants  $c_i$, $1 \leq i \leq k$,  such that $\prod_{i=1}^k  c_i =1$, and such that for all $1 \leq i \leq k$, 
\begin{enumerate}
\item
\begin{equation*} \label{gleich1}
f_i(x) = c_i\, \rho^{\frac{1}{k}} \left( \frac{k(k-1)}{2} |x |^2\right)
\end{equation*}
almost everywhere on $\mathbb{R}^n$.
\item
The function $\rho$ satisfies  the inequality 
\begin{equation*}
\prod_{i=1}^k \rho^{\frac{1}{k}} \left( \frac{k(k-1)}{2} \, |x_i  |^2 \right)
\le   \rho \left( \sum_{i, j=1, i < j}^{k} \langle x_i, x_j \rangle \right)
\end{equation*}
for all $x_i, x_j$ in $\mathbb{R}^n_+$.
\end{enumerate}
}
\par
\noindent
Our proof uses the  Pr\'ekopa--Leindler  inequality for many functions and an  exponential change of variables as an intermediate step.
\vskip 2mm
\noindent
 The above theorem and the affine isoperimetric inequality of  affine surface area for log-concave functions of \cite{CFGLSW}  lead to 
multi-functional affine isoperimetric inequalities for log-concave functions, which we also prove in this section.
\vskip 3mm
\noindent
In Section \ref{equality} we study equality cases for unconditional functions and prove the above stated  equality characterizations. 
To do so, we need equality characterizations in the  Pr\'ekopa--Leindler  inequality. We could not find such characterizations in the literature and therefore  give a proof of those.
\vskip 3mm
\noindent
In Section \ref{manysets} we prove a  generalization of the  Blaschke--Santal\'o  inequality which involves more than two convex bodies. There, $ \|\cdot\|_{K}$ denotes the norm 
with the convex body $K$ as unit ball.
 \vskip 2mm
\noindent
{\bf Theorem \ref{BS-sets2}} {\em Let $K_i$, $1 \le i \le k$, be  unconditional convex bodies  in $\mathbb{R}^n$ 
such that 
$$
\prod_{i=1}^k  e^{- \frac{1}{2} \|x_i\|^2_{K_i}}\leq \rho \left( \sum_{i,j=1, i<j}^k \langle x_i, x_j\rangle \right) \hskip 2mm \text{for every} \hskip 2mm x_i, x_j \in \mathbb{R}_+^n, 
$$
where $\rho$ is a  positive non-increasing function $[0, \infty)$ such that  $\int_{\mathbb{R}}\rho^{\frac{1}{k}}(t^2) dt <\infty$. 
Then 
\begin{eqnarray*} 
\prod_{i=1}^k \text{\rm vol}_n(K_i) \leq \left( \frac{\text{\rm vol}_n(B^n_2)}{(2\pi)^\frac{n}{2}} \right)^k \left(  \int _{\mathbb{R}^n}  \rho^{\frac{1}{k}} \Bigl( \frac{k(k-1)}{2} \, |x|^2 \Bigr) dx\right)^k.
\end{eqnarray*}
For $k >2$, equality holds 
if and only if   
$K_i=r \,B_2^n$ and $\rho(t)= e^{- \frac{t}{(k-1)r^2}}$ for some $r>0$.
 \par
\noindent
In particular, if $\rho(t)= e^{- \frac{t}{k-1}}$, then, if $\sum_{i=1, i< j }^k \langle x_i, x_j\rangle \leq \frac{k-1}{2} \sum_{i=1}^k \|x_i\|^2_{K_i}$, 
we have that 
\begin{equation*}
\prod_{i=1}^k \text{\rm vol}_n(K_i) \le  \left(\text{\rm vol}_n(B^n_2)\right)^k
\end{equation*}
and for $k >2$ equality holds  if and only if $K_i=B^n_2$ for all $1 \leq i\leq k$.
 }
\vskip 3mm
\noindent
{Proposition \ref{asp-sets} of this section gives a  version of the $L_p$-affine isoperimetric inequalities for many sets.}
\vskip 3mm
\noindent
In Section \ref{transport}
we prove several  strengthenings of classical inequalities using barycenters. Among them is the following 
``pointwise Pr\'ekopa--Leindler inequality''.
  \vskip 2mm
\noindent
{\bf Theorem \ref{ppli}}
{\em Let $\mu$ be the barycenter of measures $\mu_i = \frac{f_i}{\int f_i dx_i} dx_i$ with weights $\lambda_i$, $1 \leq i \leq k$, where $f_i$ are nonnegative integrable functions. Then it has density $p$ satisfying
\begin{equation}
\label{PBL}
\prod_{i=1}^k \Bigl( \int f_i dx_i \Bigr)^{\lambda_i} p(x) \le \sup_{x = \sum_{i=1}^k \lambda_i y_i } \prod_{i=1}^k f^{\lambda_i}_i(y_i), \hskip 2mm \text{for} \hskip 2mm p-a.e. \, x.
\end{equation}
}
\vskip 3mm
\noindent
In Section \ref{Talagrand} we study applications of our results  to transportation inequalities for the barycenter problem.
We obtain the following bound which generalizes (\ref{Fathi-intro}) and, in particular, a classical estimate of Talagrand.
\vskip 2mm
\noindent
{\bf Theorem \ref{BarycenterTheo}}
{\em Assume that $\mu_i = \rho_i \cdot \gamma$, where $\gamma$ is the standard Gaussian measure
and  the $\rho_i$ are unconditional, then
$$
\mathcal{F}(\mu) \le  \frac{k-1}{k^2} \sum_{i=1}^k \int \rho_i \log \rho_i d \gamma,
$$
 where $\mathcal{F}(\mu) = \frac{1}{2k} \sum_{i=1}^k W^2(\mu,\mu_i) $ and   $\mu$ is the barycenter of $\{\mu_i\}$ with weights $\frac{1}{k}$.}
\par
\noindent
Moreover, from our refinement of the Pr\'ekopa--Leindler inequality, we  deduce some new  inequalities related to displacement convexity of the Gaussian entropy.  
\vskip 3mm
\noindent
In Section \ref{monotone} we prove a monotonicity property of the multimarginal Blaschke--Santal\'o functional. 
A simplified version of the result is stated next.
\vskip 2mm
\noindent
{\bf Theorem  \ref{monotoneBS}}
{\em Assume that for $1 \leq i \leq k$, $V_i(x_i)$ are measurable functions such that $e^{-V_i}$ are integrable, satisfying
$$
\sum_{i =1}^k  V_i(x_i) \ge  \sum_{i, j=1, i<j}^k \frac{1}{k-1}\langle x_i , x_j \rangle. 
$$
Let the tuple of functions $ U_i(x_i)$ be the solution to the dual multimarginal maximization problem
with marginals $\frac{e^{-V_i} dx_i}{\int e^{-V_i} dx_i}$
and the cost function
$\frac{1}{k-1} \sum_{i,j=1, i < j}^k \langle x_i, x_j \rangle.$
Then
$$
\prod_{i=1}^k \int e^{-V_i} dx_i  \le 
\prod_{i=1}^k  \int e^{-U_i} dx_i.
$$
}

\section{Integral bounds and facts about barycenters} \label{IntegralBounds}

\vskip 2mm
\noindent
We start this section with the proof that the Blaschke--Santal\'o functional is bounded on the set of even functions. 
We will need the definition of the Legendre conjugate $V^*$, which for a  proper (not identically equal to $ + \infty$) 
function $V : \mathbb{R}^n \to \mathbb{R} \cup {\{+\infty\}}$ is defined as 
$$
V^*(y) = \sup_{x \in \mathbb{R}^n}( \langle x, y \rangle - V(x)).
$$
\vskip 2mm
\begin{proposition}  Let $V_i$, $1 \le i \le k,$ be a family of   Borel functions on $\mathbb{R}^n$ such that $e^{-V_i}$ is integrable for all $1 \leq i \leq k$. Then
the functional
$$
\mathcal{S}(V_1, \cdots, V_k) = \prod_{i=1}^k \int e^{-V_i(x_i)} \, dx_i 
$$
is bounded on the set  
\begin{eqnarray*}
&&\hskip -7mm L_{n,k} = \\
&& \hskip -7mm \Bigl\{ (V_1, \cdots, V_k) :  \forall i \in \{1, \cdots, k\},  V_i  \ {\rm  is \ even},  \int e^{-V_i(x)} dx < \infty,  \sum_{i=1}^k V_i(x_i) \ge  \sum_{i,j =1, i < j}^k \langle x_i, x_j \rangle \Bigr\}. 
\end{eqnarray*}
\end{proposition}
\vskip 2mm
\begin{proof}
Let us fix arbitrary  finite $(V_1, \cdots, V_k)  \in L_{n,k}$ and estimate $\mathcal{S}(V_1, \cdots, V_k)$.
First we note that the functions $V_i$ can be assumed to be convex.
Indeed, if $V_1$ is not convex, replace it by the following convex function
$$
\tilde{V}_1 (x_1) = 
 \sup_{x_i, i \ne 1} \Bigl( \sum_{i,j=1, i<j}^k \langle x_i , x_j \rangle - \sum_{i \ne 1} V_i(x_i)  \Bigr).
$$
The tuple $(\tilde{V}_1, \cdots, V_k)$ belongs to $L_{n,k}$.  Note that all the desired properties can be easily checked except of integrability of $e^{-\tilde{V}_1}$. We will show below that $\tilde{V}_1$ is integrable.  Since $V_1 \ge \tilde{V}_1$, we get  $\mathcal{S}(\tilde{V}_1, \cdots, V_k) \ge \mathcal{S}(V_1, \cdots, V_k) $.
Next we apply the same procedure to the tuple $(\tilde{V}_1, \cdots, V_k)$ and the function $V_2$. Repeating this procedure, 
we finally obtain a tuple $(\tilde{V}_1, \cdots, \tilde{V}_k)$ consisting of only convex functions such that $\mathcal{S}(\tilde{V}_1, \cdots, \tilde{V}_k) \ge \mathcal{S}(V_1, \cdots, V_k) $. Let us denote this new tuple 
again by $({V}_1, \cdots, {V}_k)$.
\par
\noindent
Next, note that without loss of generality we can restrict ourself to the case 
of convex functions satisfying $V_i(0)=0$. Indeed, one can replace $V_i$ by
$\tilde{V}_i(x_i) = V_i(x_i) - V_i(0)$, $1 \le i \le k-1$, and $V_k$ by $\tilde{V}_k(x_k) = V_k(x_k) + V_1(0) + \cdots + V_{k-1}(0)$
and this replacement does not influence the value of the integral functional. One has
$\tilde{V}_i(0)=0, \ 1 \le i \le k-1$. Next we note that
$$
\sum_{i=1}^{k-1} \tilde{V}_i(x_i)  \ge \sum_{i < j}^{k-1} \langle x_i, x_j \rangle + 
 \left\langle \sum_{i=1}^{k-1} x_i, x_k \right\rangle  - \tilde{V}_k(x_k),  \, \, \text{for all} \, \, x_k, 
$$
is equivalent to
$$
\sum_{i=1}^{k-1} \tilde{V}_i(x_i)  \ge \sum_{i < j}^{k-1} \langle x_i, x_j \rangle 
+ (\tilde{V}_k)^*\Bigl(\sum_{i=1}^{k-1} x_i \Bigr), 
$$
which in turn is equivalent to 
$$
\sum_{i=1}^{k-1} \frac{|x_i|^2}{2} + \tilde{V}_i(x_i)  \ge \frac{1}{2}\,  \left | \sum_{i =1}^{k-1}  x_i \right |^2 
+ (\tilde{V}_k)^*\Bigl(\sum_{i=1}^{k-1} x_i \Bigr).
$$
We now  define a  function $F$ by the following relation
$$
\frac{|t|^2}{2} + F(t) = \inf_{t = \sum_{i=1}^{k-1} x_i}  \sum_{i=1}^{k-1} \frac{|x_i|^2}{2} + \tilde{V}_i(x_i) .
$$
Clearly $(\tilde{V}_k)^* \le F$, hence $\tilde{V}_k \ge F^*$. Thus $\mathcal{S}(\tilde{V}_1, \cdots, \tilde{V}_k)
\le \mathcal{S}(\tilde{V}_1, \cdots, F^*)$. Moreover, it follows immediately from the definition of $F$ and the above inequalities 
that $(\tilde{V}_1, \cdots, F^*) \in L_{n,k}$. Since $\tilde{V}_i \ge 0$ and $\tilde{V}_i(0)=0$, we immediately get
$F(0)=0$. Hence, $F^*(0)=0$. Thus the tuple $(\tilde{V}_1, \cdots, F^*)$
satisfies $\tilde{V}_1(0) =  \cdots = F^*(0) =0$ and gives a larger value to $\mathcal{S}$.
\par
\noindent
Finally,  it is sufficient to show that   $\mathcal{S}$ is bounded  for finite convex even functions $V_i$ satisfying $V_i(0)=0$ and  $\sum_{i=1}^k V_i(x_i) \ge  \sum_{i,j =1, i < j}^k \langle x_i, x_j \rangle$ . We observe that for every  $j \ne m$
\begin{eqnarray*}
&&\hskip -5mm V_m(x_m) \ge 
\sup_{x_i, x_s , i, s \ne m} \Bigl( \sum_{i,s=1, i<s}^k \langle x_i , x_s \rangle - \sum_{i \ne m} V_i(x_i)  \Bigr) \geq \\
&&   \sup_{x_{j}} \Bigl( \Bigl[ \sum_{i,s=1, i<s}^k \langle x_i , x_s \rangle - \sum_{i \ne m} V_i(x_i)  
\Bigr]_{x_i=0, i \ne m, i \ne j}\Bigr)
 = \sup_{x_{j}} \bigl( \langle x_m, x_{j}  \rangle  - V_{j}(x_{j}) \Bigr) = V_{j}^*(x_m).
\end{eqnarray*}
 If $e^{-V_j}$ is integrable, then
 by the functional Blaschke--Santal\'o inequality
$$
\int e^{-V_m} dx_m \int e^{-V_j} dx_j \le  \int e^{-V_j^*} dx_j \int e^{-V_j} dx_j \le (2 \pi)^{n}.
$$
Hence
$$
\prod_{i=1}^k \int e^{-V_k} dx_k  = \Bigl( \prod_{i,j =1, i <j}^k \int e^{-V_i} dx_i \int e^{-V_j} dx_j \Bigr)^{\frac{1}{k-1}}
\le (2 \pi)^{\frac{n}{k-1}}.
$$
If  $e^{-V_j}$  is not integrable, then again by the Blaschke--Santal\'o inequality $\int e^{-V_j^*} dx_j =0$, hence
$\int e^{-V_m} dx_m  =0$, but this contradicts to finiteness of $V_m$. 
\end{proof}
\vskip 3mm
\noindent
A  related natural question  is whether there is a non-trivial lower bound for $\mathcal{S}$?
For the case of two functions this is a functional variant of the well-known open problem, known as
Mahler's conjecture. More precisely, for $k=2$ we are looking for the lower bound of the functional
$$
\int e^{-V} dx \int e^{-V^*} dy.
$$
It is conjectured  that the minimum is reached,  in particular,  when 
$V(x) = \|x\|_1 = \sum _{i=1}^n |x_i|$ or  $V(x) =\|x\|_\infty = \max_{1 \leq i \leq n} |x_i|$, or their Legendre transform. 
See e.g., M.~Fradelizi and M.~Meyer  \cite{FradeliziMeyer2008}, \cite{FradeliziMeyer2008(2)},  where the conjecture was proved 
in dimension $1$.
\vskip 2mm
\noindent
The natural generalization of this problem for the case of $k>2$ functions however has  a trivial solution.
\vskip 3mm
\noindent
\begin{proposition}
There exist even functions $V_1, V_2, V_3$ such that the triple $(V_1, V_2,V_3)$ satisfies 
\begin{equation}
\label{leg-ass}
V_m(x_m) = \sup_{x_i, i \ne m} \Bigl(  \sum_{i,j=1, i<j}^k \langle x_i , x_j \rangle - \sum_{i \ne m} V_i(x_i)  \Bigr)
\end{equation}
 and  $\mathcal{S}(V_1, V_2, V_3)=0$.
\end{proposition}
\begin{remark}
Assumption (\ref{leg-ass}) seems to be a natural generalization for $k>2$ functions of the condition that two convex functions are related by the Legendre transform. 
\end{remark}
\begin{proof}
The desired functions are
\begin{equation*}
    V_1(x_1) = 
    \begin{cases}
      0 & \text{if } x=0\\
      +\infty & \text{else}
    \end{cases}
\end{equation*}
$$
V_2(x_2) = \frac{|x_2|^2}{2}, \ V_3(x_3) = \frac{|x_3|^2}{2}.
$$
The reader can easily check the claim.
\end{proof}
\vskip 3mm
\noindent
At the end of this section we recall  basic facts on duality relations  for the transportation cost appearing in the theory
of barycenters of measures. 
Recall that for a given  family of probability measures
$\mu_1, \cdots, \mu_k$ and weights $\lambda_i \in [0,1]$ satisfying $\sum_{i=1}^k \lambda_i =1$
its barycenter  $\mu$ is the minimum point 
of the functional 
$$
\mathcal{F}(\nu) = \frac{1}{2} \sum_{i=1}^k \lambda_i W^2_2(\mu_i, \nu).
$$
Here, 
$$
W^2_2(\nu_1, \nu_2) = \inf\left\{ \int|x-y|^2 d P(x,y): P\in \mathcal{P} (\mathbb{R}^n \times \mathbb{R}^n), P(\cdot, \mathbb{R}^n) = \nu_1, P(\mathbb{R}^n, \cdot) = \nu_2\right\}
$$
is the $L^2$ Kantorovich distance of probability measures $\nu_1, \nu_2$.
It is well-known that the barycenter problem is closely related to the multimarginal  (maximization) Kantorovich problem with the cost function
$$
(x_1, \cdots, x_k) \longmapsto \sum_{i,j=1,  i \neq j}^k \lambda_i \lambda_j \langle x_i, x_j \rangle
$$
and marginals $\mu_i$. Let $\pi$ be the solution to this problem, i.e. a measure that gives a maximum to the functional
\begin{equation}
\label{multMK}
P \to  \int \sum_{i,j=1,  i \neq j}^k \lambda_i \lambda_j \langle x_i, x_j \rangle d P
\end{equation}
among the measures on $(\mathbb{R}^n)^k$ having $\mu_1, \dots,  \mu_k$ as marginals.
\vskip 2mm
\noindent
The following facts are collected from {\cite{AguehC}} and \cite{GangboSw}.

\begin{theorem} {\cite{AguehC}, \cite{GangboSw}}
\label{barycenter-th}
Assume that $\mu_i$ are absolutely continuous measures with finite second moments and $\lambda_i \in (0,1)$  are numbers 
satisfying $\sum_{i=1}^k \lambda_i =1$. Then the following facts hold.
\begin{enumerate}
\item There exists a unique absolutely continuous solution $\mu$ to the barycenter problem and a unique solution $\pi$ to the problem (\ref{multMK}).
\item The measure $\mu$ is the push-forward measure of $\pi$ under the mapping $T(x_1, \cdots, x_k) = \sum_{i=1}^k \lambda_i x_i$
and the following relation holds:
$$
 \sum_{i=1}^k \lambda_i W^2_2(\mu_i,\mu)
= \int \sum_{i=1}^k \lambda_i \, |x_i - T(x)|^2 d \pi.
$$
\item
The optimal transportation mappings $\nabla \Phi_i$ of $\mu$ onto $\mu_i$  satisfy
$$
\sum_{i=1}^k \lambda_i \nabla \Phi_i (x) = x
$$
for $\mu$-a.e. $x$.
and $\pi$ is supported on the set $\left\{(\nabla \Phi_1(x), \cdots, \nabla \Phi_k(x)) : x \in \mathbb{R}^n\right\}$.
\item
There exists  a tuple of convex functions $(v_i)$ solving the problem dual to   (\ref{multMK}), which is unique up to addition of constants and modification of sets of zero measure, i.e. a $k$-tuple of functions satisfying 
$$
 \sum_{i=1}^k v_i(x_i) \ge  \sum_{i,j, i \neq j}^k \lambda_i \lambda_j \langle x_i, x_j \rangle 
$$
with equality $\pi$-a.e.
The following relation holds between $v_i$ and $\Phi_i$:
\begin{equation}
\label{phivi}
\Phi^*_i(x_i) = \lambda_i \frac{|x_i|^2}{2} + \frac{v_i(x_i)}{\lambda_i} + C_i
\end{equation}
for $\mu_i$-almost all $x_i$.
\end{enumerate}
\end{theorem}
\vskip 2mm
\noindent
\begin{remark}
The results of item 1. are obtained in Section 3 of \cite{AguehC}, item 2. is contained in Proposition 4.2 \cite{AguehC}, item 3. corresponds to Proposition 3.8 of \cite{AguehC}. Formula (\ref{phivi}) needs some explanations.
It corresponds to formula (4.8) in \cite{AguehC}, but  in the presentation in \cite{AguehC} there is no direct link to the optimal transportation of the barycenter
$\mu$ onto   $\mu_i$. Let us give some informal explanations.
\par
\noindent
By the Kantorovich duality $\pi$ is concentrated on the zero set of the positive function
$$
 \sum_{i=1}^k v_i(x_i) - \sum_{i,j, i \neq j}^k \lambda_i \lambda_j \langle x_i, x_j \rangle. 
$$ 
Thus,  for $\pi$-a.e. $(x_1, \cdots, x_k)$  and all $1 \le i \le k$ one has
$
\nabla v_i(x_i) = \sum_{j \ne i }^k \lambda_i \lambda_j x_j. 
$
Equivalently,
\begin{equation}
\label{lvl}
\lambda_i x_i + \frac{\nabla v_i(x_i)}{\lambda_i} = \sum_{j = 1 }^k \lambda_j x_j,  \hskip 3mm \pi-a.e.
\end{equation}
It remains to note that $\mu$ is the image of $\pi$ under $T = \sum_{j = 1 }^k \lambda_j x_j $
and $\mu_i$ is the projection of $\pi$ onto the $i$-th factor. Thus relation  (\ref{lvl}) immediately implies that
$\mu$ is the image of $\mu_i$ under the mapping  $x_i \to \lambda_i x_i + \frac{\nabla v_i(x_i)}{\lambda_i}$.
Since the latter is the gradient of  the convex function $\lambda_i \frac{|x_i|^2}{2} + \frac{v_i(x_i)}{\lambda_i}$, we conclude
by uniqueness of the optimal transportation mapping  that $\nabla \Phi^*_i = \lambda_i x_i + \frac{\nabla v_i(x_i)}{\lambda_i}$.
\end{remark}
\vskip 3mm

\section{The unconditional case} \label{UCC}

\vskip 2mm
\noindent
In this section we verify our conjecture (inequality part) for the unconditional functions.
A function $f: \mathbb{R}^n \to \mathbb{R}$ is called unconditional, if 
$$
f(\varepsilon_1 x_1, \cdots, \varepsilon_n x_n) = f( x_1,  x_2, \ldots, x_n), 
$$
for every $(\varepsilon_1, \cdots, \varepsilon_n) \in \{-1,1\}^n$ and every $(x_1, \cdots, x_n) \in \mathbb{R}^n$.
\vskip 2mm
\noindent
\begin{theorem}
\label{BSunc}
Let $f_i \colon \mathbb{R}^n \to \mathbb{R}_{+}$, $1 \le i \le k$, be  measurable  unconditional integrable functions satisfying 
$$
\prod_{i=1}^k f_i(x_i)\le  \rho \left( \sum_{i, j=1, i < j}^{k} \langle x_i, x_j \rangle \right) \hskip 2mm \text{for every} \hskip 2mm x_i, x_j \in \mathbb{R}^n_+,  
$$
where $\rho$ is a  positive non-increasing function on $[0,\infty)$ such that $\int_{\mathbb{R}}\rho^{\frac{1}{k}}(t^2) dt <\infty$.  Then
\begin{equation} \label{BSunc-inequality}
 \prod_{i=1}^k \int_{\mathbb{R}^n} f_i(x_i) dx_i
\le \left( \int_{\mathbb{R}^{n}}  \rho^{\frac{1}{k}} \Bigl( \frac{k(k-1)}{2} \, |u|^2 \Bigr) du\right)^k.
\end{equation}
In particular, if  
$$
\prod_{i=1}^k f_i(x_i)\le  e^{- \alpha \sum_{i, j=1, i < j}^{k} \langle x_i, x_j \rangle }, \ \alpha \in \mathbb{R}_{+},
$$
then
\begin{equation*}
\label{fghexp}
 \prod_{i=1}^k \int_{\mathbb{R}^n} f_i(x_i) dx_i
\le \Bigl( \int_{\mathbb{R}^{n}}  e^{- \alpha \frac{k-1}{2}|u|^2} du\Bigr)^k.
\end{equation*}
\end{theorem}
\vskip 2mm
\noindent
\begin{proof}
Clearly, for unconditional functions it is sufficient to check that
$$
 \prod_{i=1}^k \int_{\mathbb{R}^n_{+}} f_i(x_i) dx_i
\le \left[ \int_{\mathbb{R}^{n}_{+}}  \rho^{\frac{1}{k}} \left( \frac{k(k-1)}{2}\,  |u|^2  \right) du \right]^k, 
$$
provided that  on $\mathbb{R}^{n}_{+}$, 
$$
\prod_{i=1}^k f_i(x_i)\le  \rho \left( \sum_{i, j=1, i < j}^{k} \langle x_i, x_j \rangle \right).
$$
We prove this using  the  Pr\'ekopa--Leindler inequality and  a trick involving a  change of variables formula (see, for instance, \cite{Gardner2002} or \cite{Lehec}, Lemma 5). 
\newline 
For $u = (u_1, \cdots, u_n)$, we denote $e^{u} = (e^{u_1}, \cdots, e^{u_n})$.
We apply  the change of variables formula
$$
x_i = e^{t_i}, \ \ t_i\in \mathbb{R}^n,
$$
and get
$$
 \prod_{i=1}^k \int_{\mathbb{R}^n_{+}} f_i(x_i) dx_i
 =
  \prod_{i=1}^k \int_{\mathbb{R}^n} f_i(e^{t_i}) e^{\sum_{m=1}^n (t_i)_m } dt_i,
$$
where we write  $t_i = \left((t_i)_1, (t_i)_2, \cdots,  (t_i)_n\right)$.
Next we apply the  Pr\'ekopa--Leindler inequality (see, e.g., \cite{Gardner2002}, formula (21) or (27)), 
$$
\prod_{i=1}^k \Bigl( \int_{\mathbb{R}^n} g_i dt_i \Bigr)^{\frac{1}{k}} \le 
 \int_{\mathbb{R}^n} \sup_{t = \frac{1}{k} \sum_{i=1}^{k} t_i} \prod_{i=1}^k  g^{\frac{1}{k}}_i(t_i) dt.
$$
After the change of  variables and the application of the Pr\'ekopa--Leindler inequality, we use the assumptions of the theorem
in the second inequality below.  We 
also use the arithmetic-geometric mean inequality and the fact that $\rho$ is non-increasing in the third inequality below.
We get
\begin{align*}
&
\Bigl( \prod_{i=1}^k \int_{\mathbb{R}^n_{+}} f_i(x_i) dx_i \Bigr)^{\frac{1}{k}}
\le \int_{\mathbb{R}^n} \sup_{t = \frac{1}{k} \sum_{i=1}^k t_i}   \prod_{i=1}^k  \Bigl( f^{\frac{1}{k}}_i(e^{t_i}) e^{\frac{1}{k} \sum_{m=1}^n (t_i)_m } \Bigr) dt
\\&
\le \int_{\mathbb{R}^n}  \sup_{t = \frac{1}{k} \sum_{i=1}^k t_i}  \Bigl[ \rho^{\frac{1}{k}} \Bigl(  \sum_{i,j =1, i < j}^{k} \sum_{m=1}^{n} e^{(t_i + t_j)_m} \Bigr) e^{\frac{1}{k} \sum_{i=1}^k \sum_{m=1}^n (t_i)_m } \Bigr] dt
\\&
= \int_{\mathbb{R}^n}  \sup_{t = \frac{1}{k} \sum_{i=1}^k t_i}  \Bigl[ \rho^{\frac{1}{k}} \Bigl(  \sum_{i,j =1, i < j}^{k} \sum_{m=1}^{n}    e^{(t_i + t_j)_m} \Bigr) \Bigr]
e^{\sum_{m=1}^n (t)_m } dt
\\&
\le \int_{\mathbb{R}^n}  \sup_{t = \frac{1}{k} \sum_{i=1}^k t_i}  \Bigl[ \rho^{\frac{1}{k}} \Bigl( \sum_{m=1}^{n}  \frac{k(k-1)}{2}   e^{  \frac{2}{k(k-1)} \sum_{i,j=1, i < j}^{k}(t_i + t_j)_m} \Bigr) \Bigr]
e^{\sum_{m=1}^n (t)_m } dt
\\&
= \int_{\mathbb{R}^n}  \sup_{t = \frac{1}{k} \sum_{i=1}^k t_i}  \Bigl[ \rho^{\frac{1}{k}} \Bigl(  \frac{k(k-1)}{2} \sum_{m=1}^{n}     e^{  \frac{2}{k} \sum_{i=1}^k (t_i)_m} \Bigr) \Bigr]
e^{\sum_{m=1}^n (t)_m } dt
 \\&
= \int_{\mathbb{R}^n}  \rho^{\frac{1}{k}}\Bigl( \frac{k(k-1)}{2} \sum_{m=1}^{n}     e^{ 2 (t)_m}\Bigr) 
e^{\sum_{m=1}^n (t)_m } dt.
\end{align*}
Changing variables $u_m = e^{(t)_m}$ one gets
$$
\Bigl( \prod_{i=1}^k \int_{\mathbb{R}^n_{+}} f_i(x_i) dx_i \Bigr)^{\frac{1}{k}}
\le \int_{\mathbb{R}^n_{+}}  \rho^{\frac{1}{k}} \Bigl( \frac{k(k-1)}{2} \sum_{m=1}^{n}     u^2_m \Bigr) du = 
\int_{\mathbb{R}^n_{+}}  \rho^{\frac{1}{k}} \Bigl( \frac{k(k-1)}{2}\,  |u|^2 \Bigr) du . 
$$
\end{proof}
\vskip 3mm
\noindent
The above theorem and the affine isoperimetric inequalities of  affine surface area for log-concave functions of \cite{CFGLSW}  lead to 
multi-functional affine isoperimetric inequalities for log-concave functions.
\par
\noindent
We first recall that for $\lambda \in \mathbb{R}$, the  $\lambda$-affine surface area 
of a convex  function $V$ was introduced in \cite{CFGLSW} as 
\begin{equation} \label{asa}
as_\lambda(V) =  \int_{\Omega_V}e^{(2\lambda-1)V(x)-\lambda \langle x, \nabla V (x)\rangle} \left(\det \, D^2V (x)\right)^\lambda dx, 
\end{equation}
where $\Omega_V=\text{int} \, (\{x \in \mathbb{R}^n: V(x) < +\infty\} )$ is  the interior of the convex domain of $V$ and 
$ D^2V$ is the Hessian of $V$.  
The gradient of $V$, denoted  by $\nabla V$, exists almost everywhere by Rademacher's theorem (see, e.g., \cite{Rademacher}),   and a theorem of Alexandrov \cite{Alexandroff} and Busemann and Feller \cite{Buse-Feller} guarantees the existence of the Hessian, denoted  by
$D^2 V$, almost everywhere in $\Omega_V$.
\vskip 2mm
\noindent
In the next theorem we collect several results  that were  shown in \cite{CFGLSW}.
\begin{theorem} \label{Th:asa} \cite{CFGLSW} Let $V: \mathbb{R}^n \to \mathbb{R} \cup \{\infty\}$ be convex. 
\vskip 2mm
\noindent
(i) \,  For any linear invertible map $A$ on $\mathbb{R}^n$, 
\hskip 5mm $as_\lambda(V \circ A) =  |\det A|^{2\lambda-1} as_\lambda (V)$.
\vskip 2mm
\noindent
(ii) \,  For all $\lambda \in \mathbb{R}$, \hskip 5mm 
$ as_\lambda(V)=as_{1-\lambda}(V^*)$. 
\vskip 2mm
\noindent
(iii)  \hskip 1mm  $as_{\frac{1}{2}} (V)  \leq  \left(\int e^{-V} dx \right)^\frac{1}{2}  \left( \int e^{-V^*} dx \right)^\frac{1}{2}$.
\vskip 2mm
\noindent
(iv) Let $V$ in addition be such that $\int _{\mathbb{R}^n} xe^{-V(x)}dx=0$, and let  $\lambda \in [0,1]$. Then
\begin{equation} \label{f-asa} 
as_\lambda(V)\leq (2\pi)^{n\lambda}\left(\int_{\mathbb{R}^n} e^{-V} dx \right)^{1-2\lambda},  
\end{equation}
and  equality holds for $\lambda\neq 0$, if and only if there exists $a\in\mathbb{R}$ and a positive definite matrix $A$ such that $V(x)=\langle Ax,x\rangle+a$, for every $x\in\mathbb{R}^n$. For $\lambda=0$, equality holds trivially.
\end{theorem}
\par
\noindent
{\bf Remark.} Theorem \ref{Th:asa} (iii) is just a special case for $\lambda=\frac{1}{2}$ of a more general statement proved in 
\cite{CFGLSW}.
\vskip 3mm
\noindent
We then get the following Proposition.
 \begin{proposition}\label{cor-f-asa}
 Let $V_i \colon \mathbb{R}^n \to \mathbb{R} \cup \{\infty\}$, $1 \le i \le k$, be convex unconditional functions and let $\rho$ be a  positive non-increasing 
function  on $[0,\infty)$ such that $\int_{\mathbb{R}}\rho^{\frac{1}{k}}(t^2) dt <\infty$.
\vskip 2mm
\noindent
(i) Let $\lambda \in [0,\frac{1}{2}]$ and suppose the $V_i$ 
satisfy
$$
\prod_{i=1}^k e^{ -V_i(x_i)}\le  \rho \left( \sum_{i, j=1, i < j}^{k} \langle x_i, x_j \rangle \right), \hskip 2mm \text{for every} \hskip 2mm x_i, x_j \in \mathbb{R}^n_+ .
$$
Then
\begin{equation} \label{multi-asa1}
 \prod_{i=1}^k as_\lambda (V_i) \leq (2\pi)^{k n\lambda} \, 
 \left( \int_{\mathbb{R}^{n}}  \rho^{\frac{1}{k}} \Bigl( \frac{k(k-1)}{2} \, |u|^2 \Bigr) du\right)^{k(1-2 \lambda)}.
\end{equation}
\par
\noindent
In particular, if  $\rho(t) = e^{-\frac{t}{k-1}}$, then 
\begin{equation}\label{multi-asa2}
 \prod_{i=1}^k as_\lambda (V_i) \leq  \left(as_\lambda \left(\frac {|\cdot |^2}{2}\right)\right)^k.
\end{equation}
\vskip 2mm
\noindent
(ii) Let $\lambda \in [\frac{1}{2}, 1]$ and suppose the $V_i$ 
are such that 
$$
\prod_{i=1}^k e^{ -V_i^*(x_i)}\le  \rho \left( \sum_{i, j=1, i < j}^{k} \langle x_i, x_j \rangle \right), \hskip 2mm \text{for every} \hskip 2mm x_i, x_j \in \mathbb{R}^n_+ .
$$
Then
\begin{equation} \label{multi-asa3}
 \prod_{i=1}^k as_\lambda (V_i) \leq (2\pi)^{k n(1-\lambda)}  \,  
 \left( \int_{\mathbb{R}^{n}}  \rho^{\frac{1}{k}} \Bigl( \frac{k(k-1)}{2} \, |u|^2 \Bigr) du\right)^{k(2 \lambda-1)}.
 \end{equation}
\par
\noindent
And again,  if  $\rho(t) = e^{-\frac{t}{k-1}}$, then 
\begin{equation}\label{multi-asa4}
 \prod_{i=1}^k as_\lambda (V_i) \leq   \left(as_\lambda \left(\frac {|\cdot |^2}{2}\right)\right)^k.
\end{equation}

\end{proposition}
\vskip 2mm
\noindent
\begin{proof} 
(i) We get immediately from Theorem \ref{BSunc} and inequality (\ref{f-asa}) that for $\lambda \in [0, \frac{1}{2}]$, 
\begin{eqnarray*} 
 \prod_{i=1}^k as_\lambda (V_i)  \leq  (2\pi)^{k n\lambda}  \,  \left( \prod_{i=1}^k  \int_{} e^{-V_i} \right)^{(1-2\lambda)} \leq (2\pi)^{k n\lambda}  \,  
 \left( \int_{\mathbb{R}^{n}}  \rho^{\frac{1}{k}} \Bigl( \frac{k(k-1)}{2} \, |u|^2 \Bigr) du\right)^{k(1-2 \lambda)}.
\end{eqnarray*}
If $\rho(t) = e^{-\frac{t}{k-1}}$, then 
$$
(2\pi)^{k n\lambda}  \,  
 \left( \int_{\mathbb{R}^{n}}  \rho^{\frac{1}{k}} \Bigl( \frac{k(k-1)}{2} \, |u|^2 \Bigr) du\right)^{k(1-2 \lambda)} = (2\pi)^{\frac{k n}{2}} =  \left(as_\lambda \left(\frac {|\cdot |^2}{2}\right)\right)^k,
 $$
which shows the second part of (i).
\vskip 2mm
\noindent
(ii) We use Theorem \ref{Th:asa} (iii) and (iv) and Theorem \ref{BSunc} and get that for $\lambda \in [\frac{1}{2}, 1]$,
\begin{eqnarray*} 
 \prod_{i=1}^k as_\lambda (V_i) &=&  \prod_{i=1}^k as_{1-\lambda} (V_i^*)  \leq  (2\pi)^{ n(1-\lambda)}  \,  \left( \prod_{i=1}^k  \int_{} e^{-V_i^*} \right)^{k(2\lambda -1)} \\
& \leq & (2\pi)^{k n(1-\lambda)}  \,  
 \left( \int_{\mathbb{R}^{n}}  \rho^{\frac{1}{k}} \Bigl( \frac{k(k-1)}{2} \, |u|^2 \Bigr) du\right)^{k(2 \lambda-1)}.
\end{eqnarray*}
The second part for $\rho(t) = e^{-\frac{t}{k-1}}$ follows.
\end{proof}
\vskip 2mm
\noindent
\begin{remark}
\par
\noindent
(i) Please note that for $\lambda=0$, inequalities (\ref{multi-asa1}) and (\ref{multi-asa2}) are just the inequalities of Theorem \ref{BSunc}. For $\lambda = \frac{1}{2}$, 
we do not need that the $V_i$ are unconditional and  the inequalities are just the inequalities of  (\ref{f-asa}),  
$$
 \prod_{i=1}^k as_{\frac{1}{2}}(V_i) \leq (2\pi)^{k n\lambda}.
$$
See also Section \ref{monotone} for more on $ as_{\frac{1}{2}}$.
\par
\noindent
(ii) For $\lambda >1$, we get an estimate from below with an absolute constant $c$, see \cite{CFGLSW}, 
$$
\prod_{i=1}^k as_\lambda (V_i) \geq c^{k n\lambda} \, 
 \left( \int_{\mathbb{R}^{n}}  \rho^{\frac{1}{k}} \Bigl( \frac{k(k-1)}{2} \, |u|^2 \Bigr) du\right)^{k(1-2 \lambda)}.
$$
\end{remark}

\section{Characterization of the equality cases} \label{equality}

In the proof of Theorem \ref{BSunc}  we have used the  Pr\'ekopa--Leindler inequality which is a particular case of the more general
 Brascamp--Lieb inequality (see \cite{BrascL}, \cite{Barthe}). To analyze the equality case we need  
the equality characterizations of the  Pr\'ekopa--Leindler inequality.
We could not find those in the  literature, except in the case of two functions, established by Dubuc \cite{Dubuc}.
 We  therefore give  a proof of  the equality characterization.

\vskip 3mm

\begin{theorem} [Pr\'ekopa--Leindler]
Let $f_i$, $1 \leq i \leq k$,  and $h$   be nonnegative integrable real functions on $\mathbb{R}^n$
such that for all $x_i$ and for all $\lambda_i \geq 0$, $ 1 \leq i \leq k$, with $\sum_{i=1}^k \lambda_i=1$, 
$$
h \left( \sum_{i=1}^{k} \lambda _i x_i \right) \geq   \prod_{i=1}^k  f^{\lambda_i}_i(x_i).
$$
Then
\begin{equation}\label{PrekopaLeindler}
\prod_{i=1}^k \Bigl( \int_{\mathbb{R}^n} f_i dx_i \Bigr)^{\lambda_i} \le 
 \int_{\mathbb{R}^n} h\,  dx.
\end{equation}
Equality holds in the Pr\'ekopa--Leindler inequality  if and only if there exist vectors $y_1, \cdots, y_k$ in $\mathbb{R}^n$ such that, 
after modification  on a set of measure zero,  the functions $f_i$  satisfy
\begin{equation} \label{Gleichung0}
\frac{f_1(x -y_1)}{\int _{\mathbb{R}^n} f_1 dx} = \frac{f_2( x -y_2)}{\int _{\mathbb{R}^n} f_2 dx} = \cdots \frac{f_k( x-y_k)}{\int _{\mathbb{R}^n} f_k dx} = e^{-\psi(x)},
\end{equation}
where $\psi$ is a convex function such that  $\int _{\mathbb{R}^n } e^{-\psi(x)} dx =1$.
 In addition, after modification on a set of zero measure, the function $h$ can be chosen to satisfy
$$h(x)= \sup_{x = \sum_{i=1}^{k} \lambda_i\, x_i} \prod_{i=1}^k  f^{\lambda_i}_i(x_i) = \prod_{i=1}^k  \Bigl( \int_{\mathbb{R}^n} f_i dx_i  \Bigr)^{\lambda_i}
e^{-\psi\bigl(x + \sum_{i=1}^k  \lambda_i y_i\bigr)}
$$
 for all $x$.
\end{theorem}
\par
\noindent
\vskip 3mm
\begin{proof}
It is clear that equality holds in inequality (\ref{PrekopaLeindler}), if the functions satisfy the condition (\ref{Gleichung0}).
\par
\noindent
The proof of the inequality is well known and can be found in,  e.g., \cite{Gardner2002, SchneiderBook}.
We give a proof of the inequality  by induction on the number of functions. This allows to establish the equality
characterizations, as for two functions, those  were established by Dubuc \cite{Dubuc}.
We have
\begin{equation*} 
 \sup_{x = \sum_{i=1}^{k} \lambda_i\, x_i} \prod_{i=1}^k  f^{\lambda_i}_i(x_i) =  \sup_{x=\lambda_1 x_1 +(1- \lambda_1)\, y}  f_1^{\lambda_1} (x_1) g^{1-\lambda_1}(y), 
\end{equation*}
where 
$$
g(y) = \sup_{y = \frac{1}{1- \lambda_1} \sum_{i=2}^{k} \lambda_i\, x_i} \prod_{i=2}^k  f^{\frac{\lambda_i}{1-\lambda_1} }_i(x_i).
$$
Applying the Pr\'ekopa--Leindler inequality for two functions gives
$$
\int 
 \sup_{x = \sum_{i=1}^{k} \lambda_i\, x_i} \prod_{i=1}^k  f^{\lambda_i}_i(x_i) \geq \left( \int f_1 dx_1 \right) ^{\lambda_1}   \left( \int g dy \right) ^{1-\lambda_1}.
 $$
Applying the induction step, one gets
$$
\int g dy \geq \prod_{i=2}^k \left(\int  f_i(x_i) dx_i \right)^{\frac{\lambda_i}{1-\lambda_1}}.
$$
This  completes the proof of the inequality.
The equality characterization follows from the equality characterization for two functions.
 \end{proof}
\vskip 3mm
\noindent
\begin{theorem}
\label{Equality}
Let $f_i \colon \mathbb{R}^n \to \mathbb{R}_{+}$, $1 \le i \le k$, be  measurable  unconditional integrable functions satisfying 
\begin{equation} \label{Mainassump}
\prod_{i=1}^k f_i(x_i)\le  \rho \left( \sum_{i, j=1, i < j}^{k} \langle x_i, x_j \rangle \right) \hskip 2mm \text{for every} \hskip 2mm x_i, x_j \in \mathbb{R}_+^n, 
\end{equation}
where $\rho$ is a  positive non-increasing function on $[0, \infty)$.  
Then for $k >2$ 
equality holds in inequality (\ref{BSunc-inequality}), i.e., 
\begin{equation*} 
 \prod_{i=1}^k \int_{\mathbb{R}^n} f_i(x_i) dx_i
= \left( \int_{\mathbb{R}^{n}}  \rho^{\frac{1}{k}} \Bigl( \frac{k(k-1)}{2} \, |x|^2 \Bigr) dx\right)^k
\end{equation*}
if and only if  there exist positive constants  $c_i$, $1 \leq i \leq k$,   such that $\prod_{i=1}^k  c_i =1$, and 
such that for all $1 \leq i \leq k$,  
\begin{enumerate}
\item
\begin{equation} \label{gleich1}
 f_i(x_i) = c_i\, \rho^{\frac{1}{k}} \left( \frac{k(k-1)}{2} | x_i |^2\right).
\end{equation}
for almost all  $x \in \mathbb{R}^n$,
\item
The function $\rho$ satisfies  the inequality 
\begin{equation} \label{gleich22}
 \prod_{i=1}^k \rho^{\frac{1}{k}} \left( \frac{k(k-1)}{2} \, |x_i|^2 \right)
\le   \rho \left( \sum_{i, j=1, i < j}^{k} \langle x_i, x_j \rangle \right)
\end{equation}
 for all $x_i, x_j$ in $\mathbb{R}^n_+$.
\end{enumerate}
\end{theorem}
\vskip 2mm
\noindent
\begin{proof}
Obviously, if (\ref{gleich1}) and  (\ref{gleich22}) hold, then one has equality in (\ref{BSunc-inequality}) and the assumption (\ref{Mainassump}) is satisfied. 
\par
\noindent
If equality holds in Theorem \ref{BSunc}, then we have equality everywhere in the proof of Theorem \ref{BSunc}. 
We have equality in 
the  Pr\'ekopa--Leindler inequality. 
Note that the  Pr\'ekopa--Leindler inequality
is applied to the functions
$$
 f_i(e^{t_i}) e^{\sum_{m=1}^n (t_i)_m }. 
$$
Hence by the above equality characterizations in the Pr\'ekopa--Leindler inequality one can modify the functions $f_i$ an a set of zero measure in such a way, that there exist  $y_1, \cdots, y_k$ such that  and all $1 \le i \le k$
\begin{equation}
 \label{20.05.2021}
 f_i(e^{t_i}) = \Bigl(\int _{\mathbb{R}^n} f_i dx \Bigr)  e^{-\sum_{m=1}^n (t_i)_m } e^{-\psi(t_i + y_i)}
\end{equation}
where $\psi$ is a convex function such that  $\int _{\mathbb{R}^n} e^{-\psi(x)} dx =1$.  In addition,
 the following equality must hold for  almost  all  $t$
\begin{align*}
 \sup_{t = \frac{1}{k} \sum_{i=1}^k t_i}  &  \prod_{i=1}^k  \Bigl( f^{\frac{1}{k}}_i(e^{t_i}) e^{\frac{1}{k} \sum_{m=1}^n (t_i)_m } \Bigr)
 =
   \sup_{t = \frac{1}{k} \sum_{i=1}^k t_i}  \Bigl[ \rho^{\frac{1}{k}} \Bigl(  \sum_{i,j =1, i < j}^{k} \sum_{m=1}^{n}    e^{(t_i + t_j)_m} \Bigr) \Bigr]
e^{\sum_{m=1}^n (t)_m }
\\& =
  \rho^{\frac{1}{k}}\Bigl( \frac{k(k-1)}{2} \sum_{m=1}^{n}     e^{ 2 (t)_m}\Bigr) 
e^{\sum_{m=1}^n (t)_m }.
\end{align*}
In particular, changing variables $x_i=e^{t_i}$
one gets 
\begin{equation}
\label{geommean}
  \sup_{(x)_m = \prod_{i=1}^k (x_i)^{\frac{1}{k}}_m}  \Bigl[ \rho^{\frac{1}{k}} \Bigl(  \sum_{i,j =1, i < j}^{k} \langle x_i,  x_j \rangle \Bigr) \Bigr]
=
  \rho^{\frac{1}{k}}\Bigl( \frac{k(k-1)}{2} |x|^2 \Bigr). 
\end{equation}
Further,  substituting (\ref{20.05.2021}),  one gets that for a.e. $t$
$$
\prod_{i=1}^k \Bigl(\int _{\mathbb{R}^n} f_i dx \Bigr)^{\frac{1}{k}}    \sup_{t = \frac{1}{k} \sum_{i=1}^k t_i}   \prod_{i=1}^k 
 e^{-\frac{1}{k} \psi(t_i + y_i)}
 =
  \rho^{\frac{1}{k}}\Bigl( \frac{k(k-1)}{2} \sum_{m=1}^{n}     e^{ 2 (t)_m}\Bigr) 
e^{\sum_{m=1}^n (t)_m}.
$$
Applying convexity of $\psi$,  one gets that $ \sup_{t = \frac{1}{k} \sum_{i=1}^k t_i}   \prod_{i=1}^k 
 e^{-\frac{1}{k} \psi(t_i + y_i)} = e^{- \psi (t+y)}$, where $y = \frac{1}{k}\sum_{i=1}^k y_i$. 
Finally,
$$
\prod_{i=1}^k \Bigl(\int _{\mathbb{R}^n} f_i dx \Bigr)^{\frac{1}{k}}    e^{- \psi (t+y)} 
=    \rho^{\frac{1}{k}}\Bigl( \frac{k(k-1)}{2} \sum_{m=1}^{n}     e^{ 2 (t)_m}\Bigr) 
e^{\sum_{m=1}^n (t)_m }
$$
 almost everywhere. Note that, if fact, equality holds pointwise, because $e^{-\psi}$ is a continuous function on $\{ \psi < \infty\}$ and $\rho$
is non-increasing.
Substitute $t = t_i + y_i - y$ into this identity. One gets
$$
\prod_{i=1}^k \Bigl(\int _{\mathbb{R}^n} f_i dx \Bigr)^{\frac{1}{k}}    e^{- \psi (t_i+y_i)} 
=    \rho^{\frac{1}{k}}\Bigl( \frac{k(k-1)}{2} \sum_{m=1}^{n}     e^{ 2 (t_i + y_i -y)_m}\Bigr) 
e^{\sum_{m=1}^n (t_i + y_i - y)_m }.
$$
Hence (\ref{20.05.2021}) implies  that for all  $t_i$, 
$$
 f_i(e^{t_i}) = \frac{\Bigl(\int _{\mathbb{R}^n} f_i dx \Bigr)}{\prod_{i=1}^k \Bigl(\int _{\mathbb{R}^n} f_i dx \Bigr)^{\frac{1}{k}}  }   \rho^{\frac{1}{k}}\Bigl( \frac{k(k-1)}{2} \sum_{m=1}^{n}     e^{ 2 (t_i + y_i - y)_m}\Bigr) e^{\sum_{m=1}^n (y_i - y)_m } .
$$
We make a  change of variables $x=e^{t_i}$ and get
$$
 f_i(x) = c_i\,   \rho^{\frac{1}{k}}\Bigl( \frac{k(k-1)}{2} \,    \left|  e^{ y_i - y} \,  x \right|^2 \Bigr),
$$
where $e^{y_i-y} x \in \mathbb{R}^n$ is defined by $(e^{y_i-y} x)_m = e^{(y_i-y)_m} (x)_m$ and where
$$
c_i=  \frac{\Bigl(\int _{\mathbb{R}^n} f_i dx \Bigr)}{\prod_{i=1}^k \Bigl(\int _{\mathbb{R}^n} f_i dx \Bigr)^{\frac{1}{k}}  }   \,  e^{\sum_{m=1}^n (y_i - y)_m }.
$$
Note that $\prod_{i=1}^k  c_i =1$.
Then we have by assumption (\ref{Mainassump})  for all  $x_i, x_j \in \mathbb{R}_+^n$, 
\begin{equation*}
 \rho \left( \sum_{i, j=1, i < j}^{k} \langle x_i, x_j \rangle \right) \geq \prod_{i=1}^k f_i(x_i) = 
 \prod_{i=1}^k \rho^{\frac{1}{k}}\Bigl( \frac{k(k-1)}{2} \,    \left|  e^{ y_i - y} \,  x_i \right|^2 \Bigr).
\end{equation*}
However, inequality 
\begin{equation}\label{Ungleichung}
\rho \left( \sum_{i, j=1, i < j}^{k} \langle x_i, x_j \rangle \right) \geq \prod_{i=1}^k \rho^{\frac{1}{k}}\Bigl( \frac{k(k-1)}{2} \,    \left|  e^{ y_i - y} \,  x_i \right|^2 \Bigr)
\end{equation}
only holds if $y_i=y$ for all $i$. To see that, note that  (\ref{Ungleichung}) holds in particular for $x_i=e^{-y_i}$ which leads to 
\begin{equation*}\label{Ungleich}
\rho \left( \sum_{i, j=1, i < j}^{k} \langle e^{-y_i}, e^{-y_j} \rangle \right) \geq \prod_{i=1}^k \rho^{\frac{1}{k}}\Bigl( \frac{k(k-1)}{2} \,    \left|  e^{  - y}  \right|^2 \Bigr)
= \rho \Bigl( \frac{k(k-1)}{2} \,    \left|  e^{  - y}  \right|^2 \Bigr)
\end{equation*}
and, as $\rho$ is decreasing, to
\begin{equation}\label{Ungleichung2}
 \sum_{i, j=1, i < j}^{k} \langle e^{-y_i}, e^{-y_j} \rangle  \leq 
  \frac{k(k-1)}{2} \,    \left |  e^{  - y}  \right|^2.
\end{equation}
Note that for $k>2$ inequality  (\ref{Ungleichung2}) only holds if $y_i=y$ for all $i$.  Indeed, by Jensen's inequality, 
\begin{eqnarray*} 
\frac{1}{\frac{k(k-1)}{2}}  \sum_{i, j=1, i < j}^{k} \langle e^{-y_i}, e^{-y_j} \rangle  &= & \sum_{m=1}^{n} \frac{1}{\frac{k(k-1)}{2}}  \sum_{i, j=1, i < j}^{k} e^{-(y_i + y_j)_m} 
\geq  \sum_{m=1}^{n}  e^{- \frac{1}{\frac{k(k-1)}{2}}  \sum_{i, j=1, i < j}^{k} (y_i + y_j)_m }\\
&=&  \sum_{m=1}^{n}  e^{- \frac{2}{k}  \sum_{i=1}^{k} (y_i)_m } = |e^{-y} |^2.
\end{eqnarray*} 
Equality in Jensen's inequality shows that thus $y_i=y$ for all $i$.
\par
\noindent
Consequently,
equality in (\ref{BSunc-inequality}) is equivalent to 
\vskip 2mm 
1.  \hskip 2mm $f_i(x) = c_i\,   \rho^{\frac{1}{k}}\Bigl( \frac{k(k-1)}{2} \,    \left| x \right|^2 \Bigr)$,
almost everywhere and 
\par
2. \hskip 2mm $\prod_{i=1}^k \rho^{\frac{1}{k}} \left( \frac{k(k-1)}{2} \, |x_i|^2 \right)
\le   \rho \left( \sum_{i, j=1, i < j}^{k} \langle x_i, x_j \rangle \right).$
\end{proof}
\vskip 3mm
\noindent
Equation (\ref{gleich1}) says in particular that if equality holds, then all $f_i$ are equal modulo normalization.
\par
\noindent
Under some natural  assumptions on the function $\rho$, one can show that 
inequality (\ref{gleich22})  always holds.
\vskip 2mm

\noindent
\begin{remark} \label{Specialrho}
Let $\rho(t) = e^{-W(t)}$, where $W$ is convex and increasing. 
Then (\ref{gleich22}) holds. 
\begin{proof}
If $\rho(t) = e^{-W(t)}$, inequality (\ref{gleich22}) is equivalent to 
$$
\frac{1}{k} \sum_{i=1}^k W \left( \frac{k(k-1)}{2} \, |x_i |^2 \right)
\ge W \Bigl( \sum_{i,j=1, i <j}^k  \langle x_i, x_j \rangle  \Bigr).
$$
By convexity of $W$, $\frac{1}{k} \sum_{i=1}^k W \left( \frac{k(k-1)}{2} \, |x_i |^2 \right)
\ge W \left( \frac{(k-1)}{2}    \sum_{i=1}^k \, |x_i |^2 \right)$. Therefore it is enough to have that 
$$
W \left( \frac{(k-1)}{2} \,  \sum_{i=1}^k |x_i |^2 \right) \geq  W \Bigl( \sum_{i,j=1, i <j}^k  \langle x_i, x_j \rangle  \Bigr)
$$
or, as $W$ is increasing,
$$
 \frac{(k-1)}{2}  \,  \sum_{i=1}^k |x_i |^2  \geq \sum_{i,j=1, i <j}^k  \langle x_i, x_j \rangle,
$$
which holds, because
$$
\sum_{i,j=1, i <j}^k  \langle x_i, x_j \rangle \le \frac{1}{2} \sum_{i,j=1, i <j}^k \Bigl( |x_i|^2 + |x_j|^2 \Bigr)
= \frac{(k-1)}{2}  \,  \sum_{i=1}^k |x_i |^2 .
$$
\end{proof}
\end{remark}
\vskip 3mm
\noindent
Theorems \ref{BSunc} and  \ref{Equality} and Remark \ref{Specialrho} immediately yield the following corollary.
\vskip 2mm
\begin{corollary} \label{Cor-Specialrho}
Let $\rho(t)= e ^{-\frac{t}{k-1}}$. 
Let $f_i \colon \mathbb{R}^n \to \mathbb{R}_{+}$, $1 \le i \le k$, be  measurable  unconditional integrable functions satisfying 
\begin{equation*} \label{mainassump}
\prod_{i=1}^k f_i(x_i)\le  e^{- \frac{1}{k-1}\, \left( \sum_{i, j=1, i < j}^{k} \langle x_i, x_j \rangle \right)} \hskip 2mm \text{for every} \hskip 2mm x_i, x_j \in \mathbb{R}_+^n.
\end{equation*} 
Then 
\begin{equation*} \label{gleich}
 \prod_{i=1}^k \int_{\mathbb{R}^n} f_i(x_i) dx_i
\leq \left( \int_{\mathbb{R}^{n}}  e^{-\frac{ |x|^2}{k}} dx\right)^k = \left(2 \pi \right)^{k \frac{n}{2}}
\end{equation*}
and  for $k >2$ 
equality holds if and only if  there exist positive constants  $c_i$, $1 \leq i \leq k$,  such that $\prod_{i=1}^k  c_i =1$ and 
such that 
for almost all  $x \in \mathbb{R}^n$,  for all $1 \leq i \leq k$,
\begin{equation*} \label{gleich1}
f_i(x_i) = c_i\, e^{-\frac{ |x_i|^2}{2}}.
\end{equation*}
\end{corollary}
\vskip 3mm
\noindent
The next proposition addresses the equality characterizations of Proposition \ref{cor-f-asa}.
\vskip 2mm
\begin{proposition}\label {equality-faii}
Let $V_i \colon \mathbb{R}^n \to \mathbb{R} \cup \{\infty\}$, $1 \le i \le k$, be convex unconditional functions and let $\rho$ be a  positive non-increasing 
function  on $[0,\infty)$ such that $\int_{\mathbb{R}}\rho^{\frac{1}{k}}(t^2) dt <\infty$.
\vskip 2mm
\noindent
(i) Let $\lambda \in [0,\frac{1}{2}]$ and suppose the $V_i$ 
satisfy
\begin{equation}\label{Annahme:fais}
\prod_{i=1}^k e^{ -V_i(x_i)}\le  \rho \left( \sum_{i, j=1, i < j}^{k} \langle x_i, x_j \rangle \right), \hskip 2mm \text{for every} \hskip 2mm x_i, x_j \in \mathbb{R}^n \hskip 2mm \text{satisfying} \hskip 2mm \langle x_i, x_j \rangle  \geq 0.
\end{equation}
Then equality holds in inequality (\ref{multi-asa1}), i.e., 
\begin{equation*} 
 \prod_{i=1}^k as_\lambda (V_i) = (2\pi)^{k n\lambda} \, 
 \left( \int_{\mathbb{R}^{n}}  \rho^{\frac{1}{k}} \Bigl( \frac{k(k-1)}{2} \, |u|^2 \Bigr) du\right)^{k(1-2 \lambda)}, 
\end{equation*}
 if and only if  for all $i$, there are $a_i \in\mathbb{R}$  
 such that for almost all $x\in\mathbb{R}^n$, 
\begin{equation}
\label{Gleich:fais1}
V_i(x_i) = c \frac{|x_i|^2}{2} +a_i,
\end{equation}
\begin{equation}
\label{Gleich:fais2}
\rho(t) = e^{- \frac{c}{k-1} t - \sum_{i=1}^k a_i}.
\end{equation}
for some $c>0$ and numbers $a_i$.
\vskip 2mm
\noindent
(ii) Let $\lambda \in [\frac{1}{2}, 1]$ and suppose the $V_i$ 
are such that 
$$
\prod_{i=1}^k e^{ -V_i^*(x_i)}\le  \rho \left( \sum_{i, j=1, i < j}^{k} \langle x_i, x_j \rangle \right), \hskip 2mm \text{for every} \hskip 2mm x_i, x_j \in \mathbb{R}^n \hskip 2mm \text{satisfying} \hskip 2mm \langle x_i, x_j \rangle  \geq 0.
$$
Then the equality characterizations  in inequality (\ref{multi-asa3}) respectively (\ref{multi-asa4}) are as in (i) with $V_i^*$ instead of $V_i$.

\end{proposition}
\vskip 2mm
\noindent
\begin{proof}
(i) It is clear that if (\ref{Gleich:fais1}) and (\ref{Gleich:fais2}) hold, then there is equality in (\ref{multi-asa1}) and the assumption (\ref{Annahme:fais}) holds.
On the other hand,  by Theorem \ref{Th:asa}, equality holds  in the first inequality of the proof of Proposition \ref{cor-f-asa}, if and only if there exist $a_i \in\mathbb{R}$ and  positive definite matrices  $A_i$  such that for every $x\in\mathbb{R}^n$, for all $1 \leq i \leq k$, 
\begin{equation} \label{FASAgleich1}
V_i(x)=\langle A_ix,x\rangle+a_i.
\end{equation}
By Theorem \ref{Equality}, equality holds  in the second  inequality of the proof of Proposition \ref{cor-f-asa}, if and only if there exist constants $c_i$, $1 \leq i \leq k$,
such that $\prod_{i=1}^k  c_i =1$ and such that  for all $1 \leq i \leq k$,
\begin{equation}  \label{FASAgleich2}
e^{-V_i(x)} = c_i\, \rho^{\frac{1}{k}} \left( \frac{k(k-1)}{2} |x |^2\right), 
\end{equation}
almost everywhere,  and 
the function $\rho$ satisfies the inequality 
\begin{equation*}
\prod_{i=1}^k \rho^{\frac{1}{k}} \left( \frac{k(k-1)}{2} \, |x_i |^2 \right)
\le   \rho \left( \sum_{i, j=1, i < j}^{k} \langle x_i, x_j \rangle \right).
\end{equation*}
It follows from (\ref{FASAgleich1}) and (\ref{FASAgleich2}) that  for almost all $x$, for all $i$
$$
e^{-\langle A_i\, x, x \rangle}  \, e^{-a_i} = c_i\, \rho^{\frac{1}{k}} \left( \frac{k(k-1)}{2} |x |^2\right).
$$
In particular, for $x=0$, we get that for all $i$, $\rho^\frac{1}{k} (0)  = \frac{e^{-a_i}}{c_i}$ and thus for all $i$
$$
e^{-\langle A_i\, x, x \rangle}  =  \rho^{-\frac{1}{k}} (0)  \, \rho^{\frac{1}{k}} \left( \frac{k(k-1)}{2} |x |^2\right).
$$
This clearly means that $A_i = \frac{c}{2}\,  Id$ for some $c>0$ and $\rho(t) = C e^{-\frac{c}{k-1} t}$ and we easily complete the proof.
\vskip 2mm
\noindent
(ii) The proof of (ii) is done in the same way. 
\end{proof}
\vskip 4mm

\section{The Blaschke--Santal\'o inequality and  the affine isoperimetric inequality for many sets } \label{manysets}

\vskip 3mm
\noindent
The classical Blaschke--Santal\'o inequality for symmetric sets can be stated in the following way, 
$$
\int_{S^{n-1}} f^n dx \int_{S^{n-1}} g^n dy \le  n^2 \left(\text{vol}_{n} (B^n_2) \right)^2 = \left(\text{vol}_{n-1}(S^{n-1})\right)^2,
$$
where $f, g$ are positive symmetric functions on $S^{n-1}$ satisfying
$$
f(x) g(y) \le \frac{1}{\langle x,y \rangle_{+}},
$$
and where for $a \in \mathbb{R}$, $a_+=\max\{a,0\}$.  Note that if $x$ and $y$ are orthogonal, 
then the right hand side of the inequality is infinite. This happens only for set of measure zero.
The latter inequality is satisfied, in particular, if
$$
f(x) = r_{K}(x), \hskip 4mm   g(y) = \frac{1}{h_{K}(y)} = r_{K^{\circ}}(y), 
$$
where $r_{K}(x)=\max\{\lambda  \geq 0: \lambda x  \in K\}$  is the radial function of 
the   convex body  $K$,  $h_{K}(y) = \sup\{\langle x,y\rangle : x \in K\} $ is the support  function of $K$
and where for a $0$-symmetric convex body $K$ with non-empty interior, 
$$
K^{\circ} = \{ y \in \mathbb{R}^n: \langle x, y \rangle  \leq 1\,  \forall x \in K\}
$$ 
is the polar body of $K$. 
We can then  write the above as follows,
$$
\text{vol}_n(K_1)\,   \text{vol}_n(K_2) \le (\text{vol}_n(B^n_2))^2,
$$
provided
\begin{equation*}
\langle x, y \rangle  \le 1, \  \forall 
x \in K_1, \,  \forall  y \in  K_2.
\end{equation*}
\vskip 3mm
\noindent
We now prove a Blaschke--Santal\'o inequality for multiple sets. We recall that a subset $K$ in $\mathbb{R}^n$ is unconditional if its characteristic function
$\1_K$ is unconditional.
\vskip 2mm
\noindent
\begin{theorem}\label{BS-sets2}
Let $K_i$, $1 \le i \le k$, be  unconditional convex bodies  in $\mathbb{R}^n$ 
such that 
$$
\prod_{i=1}^k  e^{- \frac{1}{2} \|x_i\|^2_{K_i}}\leq \rho \left( \sum_{i,j=1, i<j}^k \langle x_i, x_j\rangle \right) \hskip 2mm \text{for every} \hskip 2mm x_i, x_j \in \mathbb{R}_+^n, 
$$
where $\rho$ is a  positive non-increasing function on $[0,\infty)$ such that  $\int_{\mathbb{R}}\rho^{\frac{1}{k}}(t^2) dt <\infty$. 
Then 
\begin{eqnarray*} 
\prod_{i=1}^k \text{\rm vol}_n(K_i) \leq \left( \frac{\text{\rm vol}_n(B^n_2)}{(2\pi)^\frac{n}{2}} \right)^k \left(  \int _{\mathbb{R}^n}  \rho^{\frac{1}{k}} \Bigl( \frac{k(k-1)}{2} \, |x|^2 \Bigr) dx\right)^k .
\end{eqnarray*}
For $k >2$, equality holds 
if and only if  
$K_i= r \, B_2^n$ and $\rho(t)= e^{- \frac{t}{(k-1)r^2}}$ for some $r>0$.
\vskip 2mm
\noindent
In particular, if $\rho(t)= e^{- \frac{t}{k-1}}$, then, if $\sum_{i=1, i< j }^k \langle x_i, x_j\rangle \leq \frac{k-1}{2} \sum_{i=1}^k \|x_i\|^2_{K_i}$, 
we have that 
\begin{equation*}
\prod_{i=1}^k \text{\rm vol}_n(K_i) \le  \left(\text{\rm vol}_n(B^n_2)\right)^k
\end{equation*}
and for $k >2$ equality holds  if and only if $K_i=B^n_2$ for all $1 \leq i\leq k$.
 \end{theorem}
 \vskip 2mm
 \noindent
\begin{proof}
As for a convex body with $0$ in its interior $\text{vol}_n(K)= \frac{\text{vol}_n(B^n_2)}{(2\pi)^\frac{n}{2}} \int _{\mathbb{R}^n} e^{- \frac{1}{2} \|x\|^2_K} dx$,
we get from Theorem \ref{BSunc} that 
\begin{eqnarray*} 
\prod_{i=1}^k \text{\rm vol}_n(K_i) = \left( \frac{\text{vol}_n(B^n_2)}{(2\pi)^\frac{n}{2}} \right)^k  \prod_{i=1}^k \int _{\mathbb{R}^n} e^{- \frac{1}{2} \|x\|^2_{K_i}} dx
\leq  \left( \frac{\text{vol}_n(B^n_2)}{(2\pi)^\frac{n}{2}} \right)^k \left(  \int _{\mathbb{R}^n}  \rho^{\frac{1}{k}} \Bigl( \frac{k(k-1)}{2} \, |x|^2 \Bigr) dx\right)^k 
\end{eqnarray*}
provided that 
$$
\prod_{i=1}^k  e^{- \frac{1}{2} \|x_i\|^2_{K_i}} \leq \sum_{i,j =1,  i<j}^k \rho \left(  \langle x_i, x_j\rangle\right).
$$
The equality characterizations follow from Theorem \ref{Equality} and Corollary \ref{Cor-Specialrho}.
Indeed, by Theorem \ref{Equality}, equality holds for $k>2$ if and only if 
there exist constants  $c_i$, $1 \leq i \leq k$,   such that $\prod_{i=1}^k  c_i =1$, and 
such that 
\vskip 2mm
1.  \,  $e^{- \frac{1}{2} \|x\|^2_{K_i}}  = c_i\, \rho^{\frac{1}{k}} \left( \frac{k(k-1)}{2} |x |^2\right)$ and 
\vskip 2mm
2. The function $\rho$ satisfies 
\begin{equation*}
\prod_{i=1}^k \rho^{\frac{1}{k}} \left( \frac{k(k-1)}{2} \, |x_i|^2 \right)
\le   \rho \left( \sum_{i, j=1, i < j}^{k} \langle x_i, x_j \rangle \right).
\end{equation*}
From the first identity we get for $x=0$  that $c_i=\frac{1}{\rho^\frac{1}{k}(0)}$  for all $i$.  As $\prod_{i=1}^k  c_i =1$, this implies that $\rho(0)=1$
and hence $c_i=1$ for all $i$.   In particular, this implies that almost everywhere on $\mathbb{R}^n$, for all $i,j$, $\|x\|_{K_i} = \|x\|_{K_j}  = \|x\|_{K}$
and thus $K_i=K$ for all $i$. 
From the relation $e^{- \frac{1}{2} \|x\|^2_{K_i}}  =  \rho^{\frac{1}{k}} \left( \frac{k(k-1)}{2} |x |^2\right)$ we get that $K_i = K = r B_2^n$, hence $e^{- \frac{t}{2r^2} }  =  \rho^{\frac{1}{k}} \left( \frac{k(k-1)}{2} t \right)$, equivalently 
$e^{- \frac{s}{r^2(k-1)}}  =  \rho \left( s \right)$. The proof is complete.
\end{proof}
\vskip 3mm
\noindent
\begin{remark}
Note that for $k=2$ the above equality characterization clearly fails: the equality    
$\text{\rm vol}_n(K) \text{\rm  vol}_n(K^{\circ}) = (\text{\rm  vol}_n(B^n_2))^2$ holds if and only if $K$ is an ellipsoid. This follows obviously from the  linear invariance
of the Blaschke--Santalo functional for two sets.
\end{remark}
\vskip 3mm
\noindent
The Blaschke--Santal\'o inequality for convex bodies is closely related to affine isoperimetric inequalities which involve the $L_p$-affine surface area.
For a convex body $K$ with centroid at $0$, and for $-\infty \leq p \leq \infty$, $p \neq -n$ it is defined as (see, e.g., \cite{Lutwak1996, SchuettWerner2004}), 
\begin{equation} \label{def:p-affine}
as_{p}(K)=\int_{\partial K}  \frac{ \kappa_K(x)^{\frac{p}{n+p}}   }{ \langle x , N_{K} (x) \rangle^{\frac{n(p-1)}{n+p}} } \
 d\mu_{ K}(x),
\end{equation}
where $\mu_K$ the Hausdorff measure on $\partial K$, the boundary of $K$, 
$N_K(x)$ is the outer unit normal at $x \in \partial K$ and $\kappa_K(x) $ is the generalized  Gauss curvature  at $x \in \partial K$.
Note   that $as_{0} (K) = n \, \text{vol}_n(K)$, and if $K$ is  $C^2_+$, then $as_{\pm \infty} (K) = n \, \text{vol}_n(K^\circ) $.
\par
\noindent
The $L_p$-affine isoperimetric inequalities state that 
for $0 \leq p \leq \infty$, 
\begin{equation}\label{pasa1}
 \frac{as_p(K)}{as_p(B^n_2)}\leq \left(\frac{\text{vol}_n(K)}{\text{vol}_n(B^n_2)}\right)^{\frac{n-p}{n+p}}
\end{equation}
and for $-n<p \leq 0$,
\begin{equation*}\label{pasa2}
\frac{as_p(K)}{as_p(B^n_2)}\geq
\left(\frac{\text{vol}_n(K)}{\text{vol}_n(B^n_2)}\right)^{\frac{n-p}{n+p}}.     
\end{equation*}
Equality holds trivially if $p=0$.
In both cases
equality  holds for $p \ne 0$ if and only if $K$ is an ellipsoid. 
If $-\infty \leq p < -n$ and $K$ is $C^2_+$, then
\begin{equation} \label{pasa3}
c^{\frac{np}{n+p}}
\left(\frac{\text{vol}_n(K)}{\text{vol}_n(B^n_2)}\right)^{\frac{n-p}{n+p}} \leq \frac{as_p(K)}{as_p(B^n_2 )},
\end{equation}
with a constant $c>0$ not depending on the dimension.
These inequalities  were proved by Lutwak \cite{Lutwak1996} for $p > 1$ and for all other $p$ by Werner and Ye
\cite{WernerYe2008}. The case $p=1$ is the classical case.
\vskip 3mm
\noindent
Theorem \ref{BS-sets2} leads to a multi-set ``affine" isoperimetric inequality.
\vskip 2mm
\noindent
\begin{proposition}\label{asp-sets}
Let $K_i$, $1 \le i \le k$, be  unconditional convex bodies  in $\mathbb{R}^n$ 
such that 
$$
\prod_{i=1}^k  e^{- \frac{1}{2} \|x_i\|^2_{K_i}}\leq \rho \left( \sum_{i=1, i<j}^k \langle x_i, x_j\rangle \right) \hskip 2mm \text{for every} \hskip 2mm x_i, x_j \in \mathbb{R}_+^n,
$$
where $\rho$ is a  positive non-increasing function on $[0,\infty)$ such that  $\int_{\mathbb{R}}\rho^{\frac{1}{k}}(t^2) dt <\infty$. 
Then we have for $0 \leq p < n$
\begin{eqnarray*} 
\prod_{i=1}^k \frac{as_p(K_i)}{as_p(B^n_2)} \leq \left( \frac{1}{(2\pi)^\frac{n}{2}} \int _{\mathbb{R}^n}  \rho^{\frac{1}{k}} \Bigl( \frac{k(k-1)}{2} \, |x|^2 \Bigr) dx\right)^{k  \frac{n-p}{n+p}}.
\end{eqnarray*}
For $k >2$, equality holds if and only if  
\vskip 2mm
1.  \hskip 2mm $K_i = r \, B^n_2$ for all $i$, where  $r>0$ is a constant, 
\vskip 2mm
2. \hskip 2mm $\rho(t) = e^{-\frac{t}{(k-1)r^2}}$. 
\vskip 3mm
\noindent
In particular,  if $\rho(t)= e^{-\frac{t}{k-1}}$ and if $\sum_{i=1, i< j }^k \langle x_i, x_j\rangle \leq \frac{k-1}{2} \sum_{i=1}^k \|x_i\|^2_{K_i}$, then 
we have that 
\begin{equation}\label{multi-set-asa}
\prod_{i=1}^k as_p(K_i) \le  \left(as_p(B^n_2)\right)^k
\end{equation}
and equality holds if and only if $K_i=B^n_2$ for all $1 \leq i\leq k$.
\vskip 3mm
\noindent
If $p=n$, then 
\begin{equation}\label{multi-set-n-asa}
\prod_{i=1}^k as_p(K_i) \le  \left(as_p(B^n_2)\right)^k
\end{equation}
and  equality holds if and only if $K_i$ is an ellipsoid for all $1 \leq i\leq k$.
 \end{proposition}
\vskip 2mm
\noindent
\begin{proof}
Let $0 \leq p \leq n$. By the affine isoperimetric inequality and Theorem \ref{BS-sets2} we get 
\begin{eqnarray} \label{K-asa-equality}
&&\prod_{i=1}^k as_p(K_i) \leq  \left(as_p(B^n_2)\right)^k  \prod_{i=1}^k \left(\frac{ \text{vol}_n(K_i)}{\text{vol}_n(B^n_2)}\right)^{\frac{n-p}{n+p}}  \nonumber \\
&&\leq \left(as_p(B^n_2)\right)^k \left( \frac{1}{(2\pi)^\frac{n}{2}} \int _{\mathbb{R}^n}  \rho^{\frac{1}{k}} \Bigl( \frac{k(k-1)}{2} \, |x|^2 \Bigr) dx\right)^{k  \frac{n-p}{n+p}}.
\end{eqnarray}
The first inequality shows that for $p=n$, 
$$\prod_{i=1}^k as_n(K_i) \leq  \left(as_n(B^n_2)\right)^k.$$
\par
\noindent
If $\rho(t)= e^{-\frac{t}{k-1}}$, then we have for all $0 \leq p \leq n$, 
\begin{eqnarray} \label{Special-K-asa-equality}
\prod_{i=1}^k as_p(K_i) \leq  \left(as_p(B^n_2)\right)^k \prod_{i=1}^k \left(\frac{\text{vol}_n(K_i)}{\text{vol}_n(B^n_2)}\right)^{\frac{n-p}{n+p}} \leq  \left(as_p(B^n_2)\right)^k
\end{eqnarray}
The equality characterizations follow from Theorem \ref{BS-sets2} and the equality characterizations of the above affine isoperimetric inequalities.
\par
\noindent
Indeed, by the affine isoperimetric inequality, equality holds in the first inequality of (\ref{K-asa-equality}) if and only if $K_i=T_i B^n_2$, where $T_i$ is a linear invertible map. By Theorem \ref{BS-sets2},  equality holds  in the second inequality of (\ref{K-asa-equality})
if and only if  
$K_i = r \, B^n_2$ for all $i$, where  $r>0$ is a constant, 
and  $\rho(t) = e^{-\frac{t}{(k-1)r^2}}$. 
\end{proof}
\vskip 2mm
\noindent
\begin{remark}  (i) For $p=n$, the inequality is just the affine isoperimetric inequality  (\ref{pasa1}). As $as_0(K) = n \text{vol}_n(K)$, the inequalities 
of the theorem for $p=0$ are just the inequalities of Theorem \ref{BS-sets2}.
\vskip 2mm
\noindent 
(ii) The corresponding inequalities for $- \infty \leq p < -n$ also hold, using  (\ref{pasa3}).
\end{remark}
\vskip 3mm
\noindent
A further multiple set version of the Blaschke--Santal\'o inequality is given in the next proposition.
\vskip 2mm
\noindent
\begin{proposition}
\label{BS-sets}
Let $K_i$, $1 \le i \le k$, be unconditional convex bodies  in $\mathbb{R}^n$ with  non-empty interior and 
radial functions $r_i= r_{K_i}$.
Assume that for all $x_i = ((x_i)_1,   \cdots, (x_i)_n) \in S^{n-1}$, 
\begin{equation}
\label{rixiinequality}
\prod_{i=1}^k r_i(x_i) \le \frac{1}{\Bigl( \sum_{j=1}^n \bigl( |(x_1)_j|^{\frac{1}{k}} \cdots |(x_k)_j|^{\frac{1}{k}} \bigr)^2  \Bigr)^{\frac{k}{2}}}. 
\end{equation}
Then
$$
\prod_{i=1}^k \text{\rm vol}_n(K_i) \le  \left(\text{\rm vol}_n(B^n_2)\right)^k.
$$ 
\end{proposition}
\vskip 2mm
\noindent
\begin{proof}
Let $m \in \mathbb{R}$, $1 \leq  m < n$ and put  $x_i = e^{t_i}$. Set $w = \frac{1}{k} \sum_{i=1}^k t_i$.
Then
\begin{equation}\label{T-Sets1}
\prod_{i=1}^k r^m_i(e^{t_i}) \, \1_{\{|e^{t_i}| \le 1\}} \,  e^{\sum_{i, j} (t_i)_j}  \le   \frac{\1_{\{|e^{w}|\le 1\} } \, e^{\sum_{i, j} (t_i)_j}  }{\Bigl( \sum_{j=1}^n e^{2 w_j}   \Bigr)^{\frac{km}{2}}}
=  \frac{\1_{\{|e^{w}| \le 1\}} \,  e^{k \sum_{j=1}^n (w)_j}}{\Bigl(  \sum_{j=1}^n e^{2 w_j}   \Bigr)^{\frac{km}{2}}}.
\end{equation}
We now apply again the change of variables $x_i=e^{t_i}$,  $1 \leq i\leq k$,  the  Pr\'ekopa--Leindler inequality and  (\ref{T-Sets1}), 
\begin{align*}
\left( \prod_{i=1}^k \int_{B^n_2 \cap {\mathbb{R}}^n_{+}} r^m_i \,  dx_i \right)^{\frac{1}{k}} &=  \left( \prod_{i=1}^k  \int_{\mathbb{R}^n}  r^{m}_i(e^{t_i}) \1_{\{|e^{t_i}| \le 1\}} \, e^{\sum_{j} (t_i)_j}   \,  dt_i\right)^{\frac{1}{k}} \\
& \leq   \int_{\mathbb{R}^n} \sup_{w=\frac{1}{k} \sum_{i=1}^k t_i} \left[\prod_{i=1}^k \left(r^\frac{m}{k}_i(e^{t_i}) \, \1_{\{|e^{t_i}| \le 1\}}\,  e^{\frac{1}{k}\sum_{i, j} (t_i)_j}\right)\right]  dw
\\& \leq   \int _{\mathbb{R}^n}\frac{\1_{\{|e^{w}| \le 1\}} \, e^{\sum_{j=1}^n (w)_j}}{\Bigl(  \sum_{j=1}^n e^{2 w_j}   \Bigr)^{\frac{m}{2}}} dw
= \int_{B^n_2  \cap {\mathbb{R}}^n_{+}}  \frac{dx}{|x|^m}.
\end{align*}
Hence by symmetry
$$
\left( \prod_{i=1}^k \int_{B^n_2} r^m_i dx_i \right)^{\frac{1}{k}} \le \int_{B^n_2}  \frac{dx}{| x |^m}.
$$
 Next we observe that every radial function $r_{i}$ satisfies
 $$
 r_{i}(x_i) = r_{i}\left(\frac{x_i}{|x_i|}\right) \frac{1}{|x_i|}.
 $$
For every $1 \leq m<n$, $m \in \mathbb{R}$,  we  introduce the finite probability measure $d\mu_m = \frac{\1_{B^n_2}(u)}{\int_{B^n_2} \frac{du}{|u|^m}} \frac{du}{|u|^m}$. 
The inequality above can then be rewritten as follows, 
$$
\prod_{i=1}^k \int_{B^n_2} r^m_i\left( \frac{x_i}{|x_i|} \right) d \mu_m \le 1.
$$
Since $\mu_m$ is rotational invariant, the above inequality can be rewritten as
\begin{equation}
\label{rmith}
\prod_{i=1}^k \int_{S^{n-1}} r^m_{i}(\theta) \, d \sigma(\theta)
\le \sigma(S^{n-1})^k,
\end{equation}
where $\sigma$ is the $(n-1)$-dimensional Hausdorff measure. 
Passing to the limit $m \to n$ and applying the Fatou's Lemma one  gets that 
(\ref{rmith}) holds for  $m=n$.
On the other hand, for $m=n$ one has for all $i$
\begin{equation}
\label{rni}
\int_{S^{n-1}} r^n_{i}(\theta) \,  d \sigma(\theta) = \sigma(S^{n-1}) \frac{\text{vol}_{n}(K_i)}{\text{vol}_{n}(B^n_2)}.
\end{equation}
From this we derive  the desired estimate.
\end{proof}

\vskip 3mm
\noindent

\section{Pr\'ekopa--Leindler and displacement convexity inequalities: refinement of the transportational argument} \label{transport}
In this section we recall  the transportational arguments of F.~Barthe \cite{Barthe}
in his proof of the reverse Brascamp--Lieb inequality.
We show that the use of barycenters gives certain refinements  of the Pr\'ekopa--Leindler inequality.
\par
\noindent
{\bf In this section  we do  not assume that the functions $f_i$ are even}. 

\vskip 3mm
\noindent
Let $f_i$, $1 \leq i \leq k$, be nonnegative integrable functions and $\lambda_i \in [0,1]$ be numbers such that  $\sum_{i=1}^k \lambda_i =1$, and let
$d \mu = p(x) dx$ be a probability measure.
For every $i$, $\nabla \Phi_i$  is the optimal transportation mapping
that  pushes forward $\mu$ onto $\mu_i=f_i \, dx$.
\vskip 2mm
\noindent
In what follows we apply the change of variables formula for the optimal transportation mapping.
In that form it was established by R. McCann (see \cite{Villani}, Theorem 4.8), 
$$
p(x) = \frac{f_i(\nabla \Phi_i)}{ \int f_i dx_i} \det D^2_a \Phi_i(x), 
$$
where $D^2_a \Phi_i$ is the absolutely continuous part of the distributional Hessian $D^2 \Phi_i$ of $\Phi_i$.
In particular, it is a nonnegative matrix-valued measure. 
This formula holds almost everywhere with respect to Lebesgue measure.
We will also apply below the following results
\begin{itemize}
\item
The arithmetic-geometric mean inequality
$$
\prod_{i=1}^k {(\det A_i )}^{\lambda_i} \le \det \Bigl( \sum_{i=1}^ k \lambda_i A_i\Bigr),
$$
where the $A_i$ are symmetric nonnegative matrices, $\lambda_i \ge 0$, $\sum_{i=1}^k \lambda_i =1$.
\item
The inequality between the distributional Hessian and its absolutely continuous part
$$
0 \le D^2_a \Phi_i \le D^2 \Phi_i.
$$ 
\end{itemize}
First, we get by the arithmetic-geometric mean inequality
\begin{eqnarray}
\label{keyinequality0}
p(x) &=&
 \prod_{i=1}^{k} \left(  \frac{f_i(\nabla \Phi_i)(x)}{ \int f_i dx_i} \det D^2_a \Phi_i(x) \right)^{\lambda_i}
 \le 
  \prod_{i=1}^{k} \left(  \frac{f_i(\nabla \Phi_i)(x)}{ \int f_i dx_i}  \right)^{\lambda_i} \det \left( \sum_{i=1}^k \lambda_i D^2_a \Phi_i (x)\right)
\nonumber \\
&\le& 
\sup_{\left\{y_i: \sum_{i} \lambda_i y_i = \sum_{i} \lambda_i \nabla \Phi_i(x) \right\}} \prod_{i=1}^{k} \left(  \frac{f_i(y_i)}{ \int f_i dx_i}  \right)^{\lambda_i} 
\det \left( \sum_{i=1}^k \lambda_i D^2_a \Phi_i (x)\right).
\end{eqnarray}
In the proof of Barthe, one fixes an \emph {arbitrary} measure $\mu$ and integrates inequality (\ref{keyinequality0}). By the change of  variables
  $y =  \sum_{i} \lambda_i \nabla \Phi_i(x) $,  we get the Pr\'ekopa--Leindler inequality
  $$
  \prod_{i=1}^{k} \left( { \int f_i dx_i}  \right)^{\lambda_i}  \le 
  \int  \sup_{\{y_i: \sum_{i} \lambda_i y_i = y \}} \prod_{i=1}^{k} {f^{\lambda_i}_i(y_i)} dy.
  $$
\vskip 3mm
\noindent
If instead of an arbitrary measure $\mu$,  we apply this result to the barycenter of the $\mu_i's$, we obtain the following
pointwise refinement of the Pr\'ekopa--Leindler inequality.
\vskip 2mm
\noindent
\begin{theorem} \label{ppli}
{\bf (Pointwise Pr\'ekopa--Leindler inequality)}
Let $\mu$ be the barycenter of the $\mu_i$ with weights $\lambda_i$. Then it has a density $p$ satisfying
\begin{equation}
\label{PBL}
\prod_{i=1}^k \left( \int f_i dx_i \right)^{\lambda_i} p(x) \le \sup_{x = \sum_{i=1}^k \lambda_i y_i } \prod_{i=1}^k f^{\lambda_i}_i(y_i),  \hskip 2mm \text{for} \hskip 2mm p-a.e. x.
\end{equation}
\end{theorem}
\par
\noindent
\begin{proof}
By the arithmetic-geometric mean inequality one has
$$
\prod_{i=1}^k \Bigl( \det D^2_a \Phi_i(x) \Bigr)^{\lambda_i} \le \det \Bigl( \sum_{i=1}^k \lambda_i D^2_a \Phi_i(x)  \Bigr).
$$
Since  $\sum_{i=1} \lambda_i \Phi_i(x) =\frac{|x|^2}{2}$ for
$p(x)\,dx$-almost all $x$,  (see Theorem \ref{barycenter-th}, 3), one gets
$$
\sum_{i=1}^k \lambda_i D^2_a \Phi_i(x) \le D^2 \bigl(\sum_{i=1} \lambda_i \Phi_i(x) \bigr) = I
$$
$p(x)\,dx$-a.e.
Using this inequality and inequality
(\ref{keyinequality0}) one gets the result.
\end{proof}
\vskip 2mm
\noindent
\begin{remark}
Following the proof,  one can easily get the equality characterization for the Pr\'ekopa--Leindler inequality. Indeed, we have  equality 
in the arithmetic-geometric mean inequality if and only if $D^2_a \Phi_i(x)$ are all equal for almost all $x$.  Next, from the relation  
$\sum_{i=1} \lambda_i \Phi_i(x) =\frac{|x|^2}{2}$ one can easily get that every $\Phi_i$ has the form $\Phi_i(x) = \frac{|x|^2}{2} + \langle x, a_i \rangle + b_i $. This easily implies that the $f_i$ differ by shifts. The rest of the proof is standard. 
\end{remark}
\vskip 3mm
\noindent
Let us rewrite (\ref{keyinequality0}) in terms of the standard Gaussian reference measure 
$d\gamma = \frac{e^{-\frac{|x|^2}{2}} }{(2 \pi)^{\frac{n}{2}}} dx$.
\vskip 2mm
\noindent
\noindent
\begin{corollary}
Let $f_i dx_i = \rho_i \cdot d\gamma$ be probability  measures and let $d\mu = \rho \cdot d\gamma$ be their barycenter.
Then $\mu$-a.e.
\begin{equation}
\label{2111}
\rho(x) \, e^{\frac{1}{2} \sum_{i=1}^k \lambda_i |\nabla \Phi_i(x) -x|^2} \le \prod_{i=1}^k  \rho^{\lambda_i}_i(\nabla \Phi_i).
\end{equation}
\end{corollary}
\par
\noindent
\begin{proof}
Applying the first inequality of (\ref{keyinequality0}) to $f_i=\rho_i \,  \frac{e^{-\frac{|x|^2}{2}} }{(2 \pi)^{\frac{n}{2}}}$ and $p=\rho \,  \frac{e^{-\frac{|x|^2}{2}} }{(2 \pi)^{\frac{n}{2}}}$, we get
$$
\rho(x) \,  e^{-\frac{|x|^2}{2}} \le   \prod_{i=1}^k  \rho^{\lambda_i}_i(\nabla \Phi_i) e^{-\lambda_i \frac{|\nabla \Phi_i|^2}{2}}.
$$
Also using Theorem \ref{barycenter-th}, 3.,  we finally observe that
\begin{align*}
\sum_{i=1}^k \lambda_i \left( \frac{|\nabla \Phi_i(x)|^2}{2} - 
\frac{|x|^2}{2}\right)
& = \sum_{i=1}^k \lambda_i \left( \frac{|\nabla \Phi_i(x)|^2}{2} - 
\frac{|x|^2}{2} - \langle \nabla \Phi_i(x) - x, x \rangle \right) 
\\& = 
\frac{1}{2} \sum_{i=1}^k \lambda_i |\nabla \Phi_i(x) -x|^2.
\end{align*}
\end{proof}
\vskip 2mm
\noindent
Integrating pointwise inequality (\ref{PBL})  we get the Pr\'ekopa--Leindler inequality.
Taking logarithm of (\ref{2111}) and integrating we get the displacement convexity property of the Gaussian entropy, 
\begin{equation}
\label{entw2}
{\rm Ent}_{\gamma} (\mu) + \frac{1}{2} \sum_{i=1}^k \lambda_i  W^2_2(\mu, \mu_i ) \le \sum_{i=1}^k \lambda_i {\rm Ent}_{\gamma} (\mu_i). 
\end{equation}
This result was proved in \cite{AguehC}.
\vskip 3mm
\noindent
Mimicking the arguments that were used in the proof of (\ref{keyinequality0})  leads to the following result.
\vskip 2mm
\noindent
\begin{theorem}
Let $f_i$, $1 \leq i \leq k$,  be integrable functions satisfying 
\begin{equation}
\label{0612ass}
\prod_{i=1}^k  f^{\lambda_i}_i (x_i) \le  g \left(\sum_{i=1}^k \lambda_i x_i \right),
\end{equation}
where $\lambda_i \in [0,1], \sum_{i=1}^k \lambda_i =1$ and $g$ is a nonegative function.
Then  for $\rho dx$-almost all $x$, 
\begin{equation}
\label{0612}
\prod_{i=1}^k \left( \int f_i dx_i \right)^{\lambda_i}  \rho(x) \le  g(x), 
\end{equation}
where $\rho(x) dx$ is the barycenter of the measures $\frac{f_i}{\int f_i dx_i} dx_i$ with weights $\lambda_i$. 
\end{theorem}
\par
\noindent
\begin{proof}
Applying inequality (\ref{keyinequality0}) and the relation $\sum_{i=1}^k \lambda_i  \nabla \Phi_i(x) =x$
one immediately gets
$$
\prod_{i=1}^k  \Bigl(\int f_i (x_i) dx_i \Bigr)^{\lambda_i} \rho(x) \le \sup_{\left\{y_i: \sum_{i} \lambda_i y_i = x \right\}} \prod_{i=1}^{k} {f^{\lambda_i}_i(y_i)} \det \bigl( \sum_{i=1}^k \lambda_i D^2 \Phi_i (x)\bigr) \le g(x).
$$
\end{proof}
\vskip 2mm
\noindent
 \begin{remark}
 Assuming  (\ref{0612ass})  and integrating   (\ref{0612}) one gets the  inequality
 \begin{equation}
 \label{110721-2}
 \prod_{i=1} \Bigl( \int f_i dx_i \Bigr)^{\lambda_i}   \le  \int g(x) dx, 
 \end{equation}
 which can be considered as a weak form of the Blaschke--Santal\'o 
 functional inequality,  because it is equivalent to (\ref{lientwi}) (see the explanations in Remark \ref{110721}), which is a weaker version of the displacement convexity property (\ref{entw2}). Inequality (\ref{110721-2}) follows, of course, directly from the Pr\'ekopa--Leindler inequality.
 \end{remark}
\vskip 2mm
\noindent
In particular, assuming that the functions $V_i$ satisfy
$$
\sum_{i=1} \lambda_i V_i(x_i) \ge \frac{1}{2}  \Bigl|\sum_{i=1}^k \lambda_i x_i \Bigr|^2,
$$
one gets
$$ \Bigl( \prod_{i=1}^k \int e^{-V_i(x_i)} \ dx_i \Bigr)^{\lambda_i} \rho(x)
\le  e^{- \frac{|x|^2}{2}}.
$$ 
Rewriting this inequality with respect to the Gaussian reference measure $\gamma$, one gets the following equivalent formulation.
\begin{corollary}
\label{gauss-wbs}
Assume that the measurable functions $F_i$ satisfy
$$
\sum_{i=1}^k \lambda_i F_i(x_i) \le  \frac{1}{2} \Bigl[ \sum_{i=1}^k  \lambda_i |x_i|^2  -  \Bigl|\sum_{j=1}^k \lambda_j  x_j \Bigr|^2 \Bigr].
$$
Then
$$
\Bigl( \prod_{i=1}^k    \int e^{F_i} d \gamma \Bigr)^{\lambda_i}   {p(x)} \le 1, 
$$
where $p \cdot \gamma$ is the barycenter of $\frac{e^{F_i}}{\int e^{F_i} d \gamma} \cdot \gamma$.
\end{corollary}

\section{Talagrand-type estimates for the barycenter functional} \label{Talagrand}

In this section we show that a weak form of the Blaschke--Santal\'o inequality is related to the displacement convexity property
 of the Gaussian entropy.
The conjectured strong form of the Blaschke--Santal\'o inequality is equivalent to a certain strong entropy-$W_2$-bound, a particular case of this bound for two functions was proved by M.~Fathi in \cite{Fathi}. 
\vskip 2mm
\noindent
Let us briefly recall the main transportation Gaussian inequalities.
\begin{enumerate}
\item Every  probability measure $f \cdot \gamma$ (not necessary centered) satisfies the Talagrand transportation inequality
$$
\frac{1}{2} W^2( f \cdot \gamma, \gamma) \le {\rm Ent}_{\gamma} f := \int f \log f d \gamma.
$$
\item
In the case when one of the measures $f \cdot \gamma,  g \cdot \gamma$ is centered, a stronger inequality holds (see Remark \ref{fathi-comments} and the comments after it)
\begin{equation}
\label{fathi-2} 
\frac{1}{2} W^2_2(f\cdot \gamma, g \cdot \gamma)
 \le \int f \log f d \gamma + \int g \log g d \gamma.
\end{equation}
\item
Displacement convexity of the Gaussian entropy for arbitrary measures $\mu_i$, $1 \le i \le k$, which states that
\begin{equation}
\label{dispentgauss}
{\rm Ent}_{\gamma} (\mu) + \frac{1}{2} \sum_{i=1}^k \lambda_i  W^2_2(\mu, \mu_i ) \le \sum_{i=1}^k \lambda_i {\rm Ent}_{\gamma} (\mu_i), 
\end{equation}
where $\mu$ is the barycenter of the $\mu_i$ with weights $\lambda_i$.
\end{enumerate}
We have seen above that (\ref{dispentgauss}) follows from Theorem \ref{ppli} (pointwise Pr\'ekopa--Leindler inequality). We show below that the following  weaker version of   (\ref{dispentgauss})
\begin{equation}
\label{dispentgauss0}
 \frac{1}{2} \sum_{i=1}^k \lambda_i  W^2_2(\mu, \mu_i ) \le \sum_{i=1}^k \lambda_i {\rm Ent}_{\gamma} (\mu_i)
\end{equation}
is equivalent to some form of the Pr\'ekopa--Leindler inequality (see Remark \ref{110721}).
\vskip 2mm
\noindent
In this section we establish the equivalence (and verify it in the unconditional case in Theorem \ref{BarycenterTheo}) between the conjectured Blaschke--Santal\'o inequality and the inequality 
$$
\frac{1}{2k} \sum_{i=1}^k W^2_2(\mu_i,\mu) \le  \frac{k-1}{k^2} \sum_{i=1}^k \text{\rm Ent}_{\gamma} (\mu_i),
$$
for symmetric measures, which is stronger than (\ref{dispentgauss0}) for the choice of weights $\lambda_i = \frac{1}{k}$ and generalizes (\ref{fathi-2}) for $k>2$.
\vskip 2mm
\noindent
In what follows,  $\pi$ denotes the solution to the multimarginal Kantorovich problem with marginals $\mu_i$.
Note that
$$
\sum_{i=1}^k  \bigl|x_i - \frac{1}{k} {\sum_{ j=1}^k x_i } \bigr|^2 = \frac{1}{k} \sum_{i,j=1, i<j}^k |x_i - x_j|^2.
$$
Hence one gets by  by Theorem \ref{barycenter-th} 
$$
\mathcal{F}(\mu) = \frac{1}{2k^2} \int \sum_{i,j=1, i<j} |x_i - x_j|^2 d \pi = \frac{1}{2k} \sum_{i=1}^k W^2_2(\mu_i,\mu).
$$
\vskip 3mm
\noindent
\begin{theorem} \label{BarycenterTheo}
Assume that for $1 \leq i \leq k$, $\mu_i = \rho_i \cdot \gamma$  are probability measures and  the $\rho_i$ are unconditional and let $\mu$ be the barycenter of the $\mu_i$  with weights $\lambda_i=\frac{1}{k}$. 
Then
\begin{equation}\label{Wasser}
\mathcal{F}(\mu) \le  \frac{k-1}{k^2} \sum_{i=1}^k \int \rho_i \log \rho_i d \gamma =  \frac{k-1}{k^2} \sum_{i=1}^k \text{\rm Ent}_{\gamma} (\mu_i).
\end{equation} 
\end{theorem}
\vskip 2mm
\noindent
\begin{proof}
Using  standard approximation arguments  and lower semicontinuity of the functional $\mathcal{F}$ one can reduce the general case to the case of compactly supported densities $\rho_i$.
By the Kantorovich duality (see e.g., \cite{Villani}), 
\begin{eqnarray*} 
\mathcal{F}(\mu) &=&  \frac{1}{2k^2} \int \sum_{i,j=1, i<j} |x_i - x_j|^2 d \pi  = \frac{k-1}{k^2}\,  \frac{1}{2(k-1)} \int \sum_{i,j=1, i<j} |x_i - x_j|^2 d \pi \\
&=& \frac{k-1}{k^2} \int \sum_{i=1}^k f_i(x_i)  d \pi = \frac{k-1}{k^2} \sum_{i=1}^k \int f_i(x_i)  d \mu_i , 
\end{eqnarray*} 
for some  measurable functions $f_i$ satisfying 
\begin{equation} \label{Gleichung-f}
\sum_{i=1}^k f_i(x_i) \le \frac{1}{2(k-1)} \sum_{i,j=1, i<j} ^k |x_i - x_j|^2,
\end{equation}
 with equality $\pi$-a.e.
\par
\noindent
Note that we can assume that the functions $f_i$ are unconditional.
Indeed, if not, replace $f_i$  for all $i$ by 
$$
g_i(x_i) = \frac{1}{2^n} \sum _{\varepsilon} f_i( \varepsilon x_i),
$$
where $\varepsilon x_i = ( \varepsilon_1 x_i^1, \varepsilon_2 x_i^2, \cdots,  \varepsilon_n x_i^n) $ and $\varepsilon_l = \pm1$, $1 \leq l \leq n$.
Then the functions $g_i$ are unconditional.  They also  satisfy the dual problem  as the measures $\mu_i$ are unconditional and as the cost function
does not change under $x_i \to \varepsilon x_i$.
\par
\noindent
Inequality (\ref{Gleichung-f}) is equivalent to
$$
\sum_{i=1}^k \bigl( f_i(x_i) - \frac{1}{2} |x_i|^2 \bigr) \le - \frac{1}{k-1}  \sum_{i,j=1, i<j} \langle x_i , x_j \rangle.
$$
We will apply  Theorem \ref{BSunc} to the functions 
$f_i(x_i) - \frac{1}{2} |x_i|^2$. To this end we need to show that $e^{f_i(x_i) - \frac{1}{2} |x_i|^2}$ are integrable functions. Moreover, let us show that $f_i(x_i) - \frac{1}{2} |x_i|^2 
\in L^{\infty}(\mu_i)$ for every $i$.
\par
\noindent
Let  $R>0$ be a number such that ${\rm{supp}}(\mu_i) \subset B_R$. Then it follows from Theorem \ref{ppli} that ${\rm{supp}}(\mu) \subset B_R$. Hence the optimal transportation mapping
$\nabla \Phi^*_i$ of $\mu$ onto $\mu_i$ satisfies the estimate $|\nabla \Phi^*_i| \le R$.  By Theorem \ref{barycenter-th}
$\Phi^*_i(x_i) = \frac{1}{2k} |x_i|^2 + k v_i(x_i) + C_i $, where $v_i$ and $f_i$ are related as follows
$$
\frac{k^2}{k-1} v_i(x_i)= \frac{1}{2} |x_i|^2 - f_i(x_i).
$$
To show that $ \frac{1}{2} |x_i|^2 - f_i(x_i) \in  L^{\infty}(\mu_i)$ it is sufficient to show that $\nabla v_i$ is bounded on the support of $\mu_i$. Indeed,
$|\nabla v_i| = \frac{1}{k} |\nabla \Phi^*_i(x_i) - \frac{1}{k} x_i | \le \frac{k+1}{k^2} R$.
\par
\noindent
It now follows from Theorem \ref{BSunc} 
that
$$
\prod_{i=1}^k \int e^{f_i(x_i) - \frac{1}{2} |x_i|^2} dx_i \le (2 \pi)^{k\frac{n}{2}}, 
$$
or, equivalently,
\begin{equation}
\label{exp1}
\prod_{i=1}^k \int e^{f_i(x_i)} d \gamma \le 1. 
\end{equation}
The claim follows from the estimates 
\begin{align*}
\int \sum_{i=1}^k f_i(x_i)  d \mu_i & \le \int \sum_{i=1}^k (f_i - \log \int e^{f_i} d \gamma ) \rho_i d \gamma
\le \sum_{i=1}^k \int \bigl( \rho_i \log \rho_i -  \rho_i + e^{\int (f_i - \log \int e^{f_i} d \gamma )} \bigr) d \gamma
\\&  = \sum_{i=1}^k \int \rho_i \log \rho_i d \gamma.
\end{align*}
Here the first inequality follows from (\ref{exp1}) and in the second inequality we apply the  inequality
$xy \le e^x + y \log y - y$, which is valid for $x \in \mathbb{R}, y \ge 0$.
\end{proof}
\vskip 1mm
\noindent
\begin{remark}
\label{fathi-comments}
This result is a generalization in the unconditional setting of a result of M.~Fathi  \cite{Fathi} for two functions:
\newline
Let $\rho_0, \rho_1$ be two Gaussian unconditional probability densities and $\rho_{1/2}$ be the corresponding barycenter.
Then inequality  \ref{Wasser} implies
\begin{align} \label{Fathi}
\frac{1}{2} W^2_2(\rho_0 \cdot \gamma, \rho_1 \cdot \gamma)
& = 2 W^2_2(\rho_0 \cdot \gamma, \rho_{1/2} \cdot \gamma )
= W^2_2(\rho_0 \cdot \gamma, \rho_{1/2} \cdot \gamma )+ W^2_2(\rho_1 \cdot \gamma, \rho_{1/2} \cdot \gamma ) \nonumber
\\& \le \int \rho_0 \log \rho_0 d \gamma + \int \rho_1 \log \rho_1 d \gamma.
\end{align}
This is a particular case of Fathi's inequality. 
\end{remark}
\vskip 2mm
\noindent
Fathi has shown that in the class of symmetric functions  inequality (\ref{Fathi}) is equivalent to a Blaschke--Santal\'o inequality involving two
exponential functions.  Already earlier,  in \cite{ArtKlarMil},  it was noted that the Blaschke--Santal\'o inequality can be re-written in terms of a 
property $\tau$ introduced by Maurey \cite{Maurey} which is dual to 
to the transportation inequality. 
We follow the approach  in \cite{Fathi} to show that  
the inequality of Theorem  \ref{BarycenterTheo}  is  also equivalent to a functional Blaschke--Santal\'o for multiple exponential functions.
\par
\noindent
Indeed, letting $\rho(t) = e^{-\frac{t}{k-1}}$ in Theorem \ref{BSunc}, we get the following multifunctional Blaschke--Santal\'o inequality:
\newline
Let $f_i \colon \mathbb{R}^n \to \mathbb{R}_{+}$, $1 \le i \le k$, be  measurable  unconditional functions with $\int e^{f_i}$ integrable  such that 
\begin{equation}\label{Bedingung-BS}
\sum_{i=1}^k f_i(x_i)\le -\frac{1}{k-1}  \sum_{i, j=1, i < j}^{k} \langle x_i, x_j \rangle.
\end{equation}
 Then
\begin{equation} \label{expBSinequality}
 \prod_{i=1}^k \int_{\mathbb{R}^n} e^{f_i}  dx
\le \left( 2 \pi \right) ^{k \frac{n}{2}}.
\end{equation}
\vskip 2mm
\noindent
\begin{proposition}
\label{BarycenterTheo2}
Inequality (\ref{Wasser})  is  equivalent to the functional Blaschke--Santal\'o inequality (\ref{expBSinequality}).
\end{proposition} 
\begin{proof}
One implication is just Theorem \ref{BarycenterTheo}. 
\par
\noindent
For the other implication, we  first  rewrite  inequality  (\ref{Wasser}). Thus, 
let $\mu$ be the barycenter of the $\mu_i= \rho_i \cdot \gamma$ with coefficients  $\frac{1}{k}$ and unconditional $\rho_i$. 
We recall  that for a probability measure $\nu$  
$$
\text{Ent}_\gamma (\nu) = \text{Ent}_{dx} (\nu) +\frac{n}{2} \log( 2\pi) + \frac{1}{2} \int|x|^2 d \nu
$$
and use this and the definition of the Kantorovich distance to get that (\ref{Wasser}) is equivalent to
\begin{equation}  \label{Wasser1} 
- \frac{2}{k} \inf _{P}  \sum_{i, j=1, i < j}^{k} \int  \langle x_i, x_j \rangle  \, dP \leq \frac{2(k-1)}{k}  \sum_{i=1}^k \text{Ent} _{dx} (\mu_i)  + (k-1) \log (2 \pi)^n.
\end{equation}
Let now the $f_i$ be unconditional and such that they satisfy (\ref{Bedingung-BS}).
We apply  (\ref{Wasser1}) to  $\mu_i = \rho_i  \gamma = \frac{e^{f_i}}{\int e^{f_i }} \gamma$. We also use that
for a probability measure $\nu$ 
\begin{equation}\label{entropy-id}
\text{Ent}_{dx} (\nu) = \sup_f \int f d \nu - \log \int e^f dx
\end{equation}
 and get 
\begin{eqnarray} \label{Wasser2}
&&\hskip -5mm - \frac{2}{k} \inf _{P}  \sum_{i, j=1, i < j}^{k} \int  \langle x_i, x_j \rangle  \, dP  \nonumber \\
&& \hskip 5mm  \leq \frac{2(k-1)}{k}  \sum_{i=1}^k \left(\int f_i d \mu_i - \log \int e^{f_i}  dx\right)   + (k-1) \log (2 \pi)^n . 
\end{eqnarray}
By the Kantorovich duality, the left hand side of this inequality equals 
\begin{eqnarray} \label{Kantor}
&&\hskip -10mm  - \frac{2}{k} \inf _{P}  \sum_{i, j=1, i < j}^{k} \int  \langle x_i, x_j \rangle  \, dP \nonumber \\
&& = \frac{2(k-1) }{k} \sup_{\sum_{i=1}^k h_i(x_i)\le -\frac{1}{k-1}  \sum_{i, j=1, i < j}^{k} \langle x_i, x_j \rangle}    \sum_{i =1}^{k} \int h_i d \mu_i \\
&&\geq  \frac{2(k-1) }{k}  \sum_{i =1}^{k} \int f_i d \mu_i \nonumber.
\end{eqnarray} 
Putting this into  (\ref{Wasser2}) and removing terms that appear on both sides gives the inequality (\ref{expBSinequality}). 
\end{proof} 
\vskip 3mm 
\noindent
\begin{remark}
\label{110721}
Mimicking the proofs of Theorem \ref{BarycenterTheo} and  Proposition \ref{BarycenterTheo2}
one can show that the  inequality 
\begin{equation}
\label{fili}
\Bigl( \prod_{i=1}^k    \int e^{F_i} d \gamma \Bigr)^{\lambda_i}    \le 1, 
\end{equation}
where  the functions $F_i$  satisfy
\begin{equation} \label{Annahme}
\sum_{i=1}^k \lambda_i F_i(x_i) \le  \frac{1}{2} \Bigl[ \sum_{i=1}^k  \lambda_i |x_i|^2  -  \Bigl|\sum_{j=1}^k \lambda_j  x_j \Bigr|^2 \Bigr],
\end{equation}
 is equivalent to the inequality
\begin{equation}
\label{lientwi}
\sum_{i=1}^{k} \lambda_i {\rm Ent}_{\gamma}(\mu_i) \ge \frac{1}{2} \sum_{i=1}^n \lambda_i W^2_2(\mu_i, \mu),
\end{equation}
where $\mu$ is the barycenter of the $\mu_i's$ with weights $\lambda_i$.
\vskip 2mm 
\noindent
Letting $V_i=  - F_i  + \frac{ |x_i|^2}{2}$ and $\lambda_i = \frac{1}{k}$, $1 \leq i \leq k$, we note  that inequality (\ref{fili}) has the following equivalent ``Euclidean'' formulation:
$$
\prod_{i=1}^k  \int e^{-V_i} dx_i  \le (2 \pi)^{k \frac{n}{2}}, 
$$
provided
$
\sum_{i=1}^k V_i(x_i) \ge \frac{1}{2k} \bigl | x_1 + \cdots + x_k \bigr |^2$. 
This inequality is a direct consequence of the Pr\'ekopa--Leindler inequality and here we do not assume that  the $V_i$ are even. 
\end{remark}
\vskip 2mm
\noindent
\begin{remark}
See also the notes to the first version of the article which contained another proof of (\ref{fili})
based on a symmetrization procedure.
\end{remark}
\vskip 2mm
\noindent 
Inequality (\ref{lientwi}) is, in fact, a weaker version of the displacement convexity property  
(\ref{entw2}). It follows, for instance, from inequality (\ref{2111}).
\vskip 3mm
\noindent
What happens,  if in the derivation of the Talagrand type bounds  instead of (\ref{fili}) one applies the stronger  pointwise inequality
$\Bigl( \prod_{i=1}^k    \int e^{F_i} d \gamma \Bigr)^{\lambda_i} p(x)    \le 1$, (see Corollary \ref{gauss-wbs})? 
The answer is given in the next theorem.
\par
\noindent
\begin{theorem}
Let $\mu_i= \rho_i \cdot \gamma$ be probability measures and $f_i(x_i)$ be the solution to the dual multimarginal problem
with marginals $\mu_i$ and the cost function $\frac{1}{2k} \sum_{i,j=1, i<j}^k |x_i - x_j|^2$.
Let $\mu = p(x) \cdot \gamma$ be the barycenter of probability measures $\frac{e^{f_i}}{\int e^{f_i} d \gamma} \cdot \gamma$ with weights $\frac{1}{k}$. 
Then 
\begin{equation}
\label{prhow}
p(x) \le e^{\frac{1}{k} \sum_{i=1}^k (\int \rho_i \log \rho_i d \gamma - \frac{1}{2} W^2_2(\mu, \mu_i))}.
\end{equation}
\end{theorem}
\begin{proof}
Let $\pi$ be the solution to the corresponding primary problem. By the  Kantorovich duality
$$
\frac{1}{2k} \sum_{i=1}^k W^2_2(\mu_i,\mu)  = \frac{1}{2k^2} \int \sum_{i,j=1, i<j} |x_i - x_j|^2 d \pi = \frac{1}{k} \sum_{i=1}^k \int f_i d \mu_i.
$$
Then
$$
\frac{1}{2k} \sum_{i=1}^k W^2_2(\mu_i,\mu) = \frac{1}{k} \sum_{i=1}^k \int f_i d \mu_i
= \log \Bigl( \prod_{i=1}^k    \int e^{f_i} d \gamma \Bigr)^{\frac{1}{k}} 
+ \frac{1}{k} \sum_{i=1}^k \int \bigl( f_i  -  \log \int e^{f_i} d \gamma )  \rho_i d \gamma.
$$
Using Corollary \ref{gauss-wbs}
one has
$$
\log \Bigl( \prod_{i=1}^k    \int e^{f_i} d \gamma \Bigr)^{\frac{1}{k}} \le - \log p(x).
$$
Then we apply Young inequality  and get that 
$
\int \bigl( f_i  -  \log \int e^{f_i} d \gamma )  \rho_i d \gamma \le \int \rho_i \log \rho_i d \gamma.
$
Finally one obtains
$
\log p(x) \le {\frac{1}{k} \sum_{i=1}^k (\int \rho_i \log \rho_i d \gamma - \frac{1}{2} W^2_2(\mu, \mu_i))}.
$
\end{proof}
\vskip 2mm
\noindent
Taking logarithm of both sides of (\ref{prhow}) and integrating with respect to $\nu = p \cdot \gamma$
we obtain, in particular, the following estimate
$${\rm Ent}_{\gamma}(\nu) + \frac{1}{2k} \sum_{i=1}^k W^2_2(\mu, \mu_i) \le \frac{1}{k} \sum_{i=1}^k  {\rm Ent}_{\gamma}(\mu_i),$$
which is reminiscent to (\ref{dispentgauss}), but it is not 
completely clear how they can be compared.

\section{Monotonicity of the Blaschke-Santal\'o functional} \label{monotone}

In this section we prove a remarkable monotonicity property of the Blaschke--Santal\'o functional  which appears naturally with respect 
to the barycenter problem. 

\subsection{The case of two functions}

We start with the case of two functions, $k=2$. We first recall for $\lambda \in \mathbb{R}$, the  definition of the $\lambda$-affine surface area 
of a convex  function $V$ introduced in \cite{CFGLSW} and already given in (\ref{asa}).
\begin{equation*}
as_\lambda(V) =  \int_{\Omega_V}e^{(2\lambda-1)V(x)-\lambda \langle x, \nabla V (x)\rangle} \left(\det \, D^2V (x)\right)^\lambda dx, 
\end{equation*}
where $ D^2V$ is the Hessian of $V$.
\vskip 3mm
\noindent
We consider now two functionals on convex functions $V$,  the  Blaschke--Santal\'o functional
$$
\mathcal{BS} (V) = \int e^{-V} dx \int e^{-V^*} dx 
$$
and the $\frac{1}{2}$-affine surface area functional,  
$$
\mathcal{J}(V) = as_{\frac{1}{2}} (V) = \int e^{- \frac{1}{2} \langle x ,\nabla V(x) \rangle } \sqrt{\det D^2 V} dx.
$$
\noindent
To avoid technicalities,  we assume that $V$ is $C^2$ and strictly convex. 
\vskip 3mm
\noindent
\begin{proposition} \label{BS+J}
Let $V$ be a strictly convex $C^2$-function such that $e^{-V}, e^{-V^*}$ are integrable functions. Let $\nabla \Psi$ be the optimal transportation of $\frac{e^{-V} dx}{\int e^{-V} dx}$
onto  $\frac{e^{-V^*} dx}{\int e^{-V^*} dx}$.
Then
\begin{equation}
\label{BSV}
\mathcal{BS}(V) \le \mathcal{J}^2(\Psi) \le 
\mathcal{BS}(\Psi).
\end{equation}
Equivalently
$$
 \int e^{-V} dx \int e^{-V^*} dx  \le \Bigl( \int  e^{-\frac{1}{2}\langle x, \nabla \Psi \rangle } {\sqrt{\det D^2 \Psi}} dx \Bigr)^2
 \le \int e^{-\Psi} dx \int e^{-\Psi^*} dx .
$$
\end{proposition}
\begin{proof}
The second inequality is just Theorem \ref{Th:asa} (iii).  To prove the first inequality, we 
apply the change of variables formula
$$
\frac{e^{-V}}{\int e^{-V} dx} = 
\frac{e^{-V^*(\nabla \Psi)}}{\int e^{-V^*} dx} \det D^2 \Psi.
$$
Note that regularity of $V, V^*$ imply that $\Psi$ is sufficiently regular, hence $D^2 \Psi$ is absolutely continuous 
(see, for instance,  \cite{Villani}).
Then
$$
 \int  e^{-\frac{1}{2}\langle x, \nabla \Psi \rangle } {\sqrt{\det D^2 \Psi}} dx
 = \sqrt{\frac{\int e^{-V^*} dx}{ \int e^{-V} dx}} \int e^{\frac{V^*(\nabla \Psi) - V(x) - \langle x, \nabla \Psi \rangle}{2}} dx.
$$
The result follows from the inequality $V^*(\nabla \Psi) + V(x) \geq \langle x, \nabla \Psi \rangle$.
\end{proof}
\vskip 3mm
\noindent
Let us outline (without rigorous justifications) the idea of alternative proof of the Blaschke--Santal\'o inequality.
It can be easily seen from the proof that equality in $V^*(\nabla \Psi) + V(x) \geq \langle x, \nabla \Psi \rangle$ (and hence in (\ref{BSV}) ) is attained if and only if
$V = \Psi + a$ for some constant $a$.  Thus, within a certain appropriate class of functions,  e.g.,  symmetric, 
the maximum of the  Blaschke--Santal\'o
functional  must satisfy that  the 
measure $\frac{e^{-\Psi^*}}{\int e^{-\Psi^*} dx}$ is the push-forward measure of 
$\frac{e^{-\Psi}}{\int e^{-\Psi} dx}$ under the mapping $\nabla \Psi$.
This means that $\Psi$ solves the following Monge--Amp\`ere equation 
\begin{equation}
\label{psipsi*}
\frac{e^{-\Psi}}{\int e^{-\Psi} dx} =\frac{e^{-\Psi^*(\nabla \Psi)}}{\int e^{-\Psi^*} dx}  \det D^2 \Psi.
\end{equation} 
It was shown in \cite{CKW} that this equation admits the following family of solutions, provided $\frac{e^{-\Psi}}{\int e^{-\Psi} dx}$
has logarithmic derivatives,
 $$\Psi = \frac{\langle Ax, x \rangle}{2}+ c, 
 $$ 
 where $A$ is a positive definite matrix and $c$ is a constant.
These are exactly the maximizers of the Blaschke--Santal\'o functional. 
\vskip 2mm
\noindent
 Thus,  this observation suggests the following (so far heuristic) approach  
to the Blaschke--Santal\'o inequality.
 Let  $\Psi_0 = V$, and consider iterations  $\Psi_{l}$, $l \in \mathbb{N}$, where  $\Psi_{l+1}$ is the optimal transportation potential
 pushing forward $\frac{e^{-\Psi_l} dx}{\int e^{-\Psi_l} dx}$
onto  $\frac{e^{-\Psi_l^*} dx}{\int e^{-\Psi_l^*} dx}$. By Proposition \ref{BS+J}, one gets an increasing sequence
$\mathcal{BS}(\Psi_l)$, $l \in \mathbb{N}$. From this,  one can try to extract convergence of $\Psi_l$ to a potential $\Psi$, which gives a maximum to
the Blaschke--Santal\'o functional.
Then prove that $\Psi$ solves (\ref{psipsi*}),  and by uniqueness deduce that $\Psi$ is quadratic.

 \subsection{The multimarginal case}

Next we generalize the previous result to the multimarginal case, $k>2$.

\begin{theorem} \label{monotoneBS}
Assume that $V_i(x_i)$, $1 \leq i \leq k$,  are measurable functions such that $e^{-V_i}$ are integrable, satisfying
$$
\sum_{i =1}^k \lambda_i V_i(x_i) \ge C \sum_{i,j=1, i < j }^k \lambda_i \lambda _j \langle x_i , x_j \rangle 
$$
for some $C>0$ and $\lambda_i \in (0,1)$ with $\sum_{i=1}^k \lambda_i =1$.
\par
\noindent
Let the tuple of functions $\lambda_i U_i(x_i)$ be the solution to the dual multimarginal maximization problem
with marginals $\frac{e^{-V_i} dx_i}{\int e^{-V_i} dx_i}$
and the cost function
$C \sum_{i,j=1, i <j }^k \lambda_i \lambda_j \langle x_i, x_j \rangle.$
Then
$$
\prod_{i=1}^k \Bigl( \int e^{-V_i} dx_i \Bigr)^{\lambda_i} \le 
\prod_{i=1}^k \Bigl( \int e^{-U_i} dx_i \Bigr)^{\lambda_i}.
$$
\end{theorem}
\vskip 2mm
\noindent
\begin{proof}
Let $\rho \,dx$ be the barycenter of $d\mu_i = \frac{e^{-V_i} dx_i}{\int e^{-V_i} dx_i}$ with weights
$\lambda_i$
and $\nabla \Phi_i$ be the optimal transportation mapping pushing forward $\rho\,  dx$ onto $d\mu_i$.
Recall that for $\rho\, dx $-almost all $y$ one has (see Theorem \ref{barycenter-th}), 
$$
\sum_{i=1}^k \lambda_i U_i(\nabla \Phi_i(y))  =  C \sum_{i,j=1, i < j}^k \lambda_i \lambda_j \langle \nabla \Phi_i(y), \nabla \Phi_j(y) \rangle.
$$
Apply the change of variables formula
$$
\rho(y) = \frac{e^{-V_i(\nabla \Phi_i(y))}}{\int e^{-V_i} dx_i} \det D^2_a \Phi_i(y).
$$
One has
\begin{align*}
\prod_{i=1}^k \Bigl( \int e^{-V_i} dx_i \Bigr)^{\lambda_i} \rho(y) 
& = e^{-\sum_{i=1}^{k} \lambda_i V_i(\nabla \Phi_i(y))} \prod_{i=1}^k \bigl( {\det} D^2_a \Phi_i(y) \bigr)^{\lambda_i}
\\&
\le e^{-  C \sum_{i,j=1, i < j}^k \lambda_i \lambda_j \langle \nabla \Phi_i(y), \nabla \Phi_j(y) \rangle} \prod_{i=1}^k \bigl( {\det} D^2_a \Phi_i(y) \bigr)^{\lambda_i}
\\&
= \prod_{i=1}^k e^{-\lambda_i U_i(\nabla \Phi_i(y))}\bigl( {\det} D^2_a \Phi_i(y) \bigr)^{\lambda_i}.
\end{align*}
Integrating both sides and using H\"older's inequality, we get
\begin{eqnarray*}
\prod_{i=1}^k \Bigl( \int e^{-V_i} dx_i \Bigr)^{\lambda_i} &\le& \int \prod_{i=1}^k e^{-\lambda_i U_i(\nabla \Phi_i)}\bigl( {\det} D^2_a \Phi_i(y) \bigr)^{\lambda_i} dy\\
&\le& \prod_{i=1}^k \Bigl( \int  e^{- U_i(\nabla \Phi_i)}\bigl( {\det} D^2_a \Phi_i(y) \bigr) dy \Bigr)^{\lambda_i}
=  \prod_{i=1}^k \Bigl( \int_{\nabla \Phi_i(\mathbb{R}^n)} e^{-U_i} dx_i \Bigr)^{\lambda_i}
\\&\le&  
\prod_{i=1}^k \Bigl( \int e^{-U_i} dx_i \Bigr)^{\lambda_i}.
\end{eqnarray*}
Here we use the change of variables and the fact that the image of ${\det} D^2_a \Phi_i(y) dy $ under $\nabla \Phi_i$ is the Lebesgue
measure on $\nabla \Phi_i(\mathbb{R}^n)$. This follows, for instance from the aforementioned  result of McCann (\cite{Villani},  Theorem 4.8).

\end{proof}
\vskip 2mm
\noindent
Let us  informally analyze the equality case.
Clearly, in this case one has { for almost all $y$, }
$$
\sum_{i=1} \lambda_i V_i(\nabla \Phi_i(y)) = 
\sum_{i=1} \lambda_i U_i(\nabla \Phi_i(y)).
$$
Integrating over $\rho\,  dy$ we get that $(\lambda_i V_i)$ is a dual Kantorovich solution as well. Hence,
 by uniqueness of the dual solution
$$
V_i = U_i + C_i, \hskip 5mm \sum_{i=1}^k C_i=0.
$$
In addition, one has for all $i$ that
$$
\frac{e^{-U_i (\nabla \Phi_i)}}{\int e^{-U_i} dx_i} \det D^2 \Phi_i = \rho,
$$
or, equivalently,
$$
\frac{e^{-U_i}}{\int e^{-U_i} dx_i} = \rho(\nabla \Phi^*_i) \det D^2 \Phi^*_i.
$$
In particular, since (see Theorem \ref{barycenter-th})  
\begin{equation*}
\Phi^*_i(x_i) = \lambda_i \frac{|x_i|^2}{2} + \frac{U_i(x_i)}{C} + C_i,
\end{equation*}
every function $U_i$ must satisfy
\begin{equation}
\label{KE2}
\frac{e^{-U_i}}{\int e^{-U_i} dx_i}  = \rho\left(\frac{\nabla U_i(x_i)}{C} + \lambda_i x_i \right)\,  \det \left(\frac{D^2 U_i }{C}  + \lambda_i I\right).
\end{equation}
Thus, a maximizer of the Blaschke--Santal{\'o} inequality, if it exists,  must satisfy the system of equations (\ref{KE2}),
where every $U_i$ is convex.
\vskip 3mm
\noindent
\begin{remark}
Equation (\ref{KE2}) is an equation of the K{\"a}hler--Einstein type. We do not know whether (\ref{KE2}) admits a unique solution.
The well posedness of the classical  K{\"a}hler--Einstein  equation 
$$
\frac{e^{-\Phi}}{\int e^{-\Phi} dx} = \rho(\nabla \Phi) \det D^2 \Phi 
$$
was  proved  under broad assumptions in  \cite{CorderoKlartag}.
\end{remark}
\vskip 3mm

\subsection*{Acknowledgement}
The authors  want to thank the referee for the careful reading and suggestions for improvement.
They also want to thank J. Lehec,  F.~Barthe, M.~Fathi, and M.~Fradelizi for valuable comments.
\par
\noindent
The second author wants to thank the Hausdorff Research Institute for Mathematics. Part of the work on the paper
was carried out during her stay there.

\vskip 2mm
\noindent
Alexander V. Kolesnikov  \\
{\small Faculty of Mathematics} \\
{\small National Research Institute Higher School of Economics} \\
{\small Moscow, Russian Federation}   \\
{\small \tt sascha77@mail.ru}

\vskip 2mm
\noindent
Elisabeth Werner\\
{\small Department of Mathematics \ \ \ \ \ \ \ \ \ \ \ \ \ \ \ \ \ \ \ Universit\'{e} de Lille 1}\\
{\small Case Western Reserve University \ \ \ \ \ \ \ \ \ \ \ \ \ UFR de Math\'{e}matique }\\
{\small Cleveland, Ohio 44106, U. S. A. \ \ \ \ \ \ \ \ \ \ \ \ \ \ \ 59655 Villeneuve d'Ascq, France}\\
{\small \tt elisabeth.werner@case.edu}\\ \\

\end{document}